\documentclass{article}

\usepackage{calc}
\usepackage{tocloft}
\usepackage{tabularx}
\usepackage{booktabs}
\usepackage{fullpage}

\usepackage{amsmath}
\usepackage{amsfonts}
\usepackage{amsthm}
\usepackage[all,2cell]{xy}
\usepackage{setspace}
\usepackage{stmaryrd}
\UseTwocells

\newcommand{\LL}{\mathbb{L}}
\newcommand{\PP}{\mathbb{P}}
\newcommand{\RR}{\mathbb{R}}
\newcommand{\TT}{\mathbb{T}}
\newcommand{\cala}{\mathcal{A}}
\newcommand{\cale}{\mathcal{E}}
\newcommand{\calo}{\mathcal{O}}

\newcommand{\frg}{\mathfrak{g}}

\newcommand{\frl}{\mathfrak{l}}

\newcommand{\frp}{\mathfrak{p}}
\newcommand{\frr}{\mathfrak{r}}

\newcommand{\HoCofib}{\operatorname{HoCofib}}
\newcommand{\HoFib}{\operatorname{HoFib}}
\newcommand{\Hom}{\operatorname{Hom}}

\newcommand{\Lagr}{\operatorname{Lagr}}
\newcommand{\Map}{\operatorname{Map}}

\newcommand{\RSpec}{\mathbb{R}\operatorname{Spec}}
\newcommand{\Spec}{\operatorname{Spec}}
\newcommand{\Sym}{\operatorname{Sym}}
\newcommand{\Symp}{\operatorname{Symp}}

\theoremstyle{plain}
\newtheorem{thm}{Theorem}[section]
\newtheorem{lem}[thm]{Lemma}
\newtheorem{cor}[thm]{Corollary}

\theoremstyle{definition}
\newtheorem{defn}[thm]{Definition}

\theoremstyle{remark}
\newtheorem{rmk}{Remark}[thm]

\begin{document}
\title{Shifted Symplectic and Poisson Structures on Spaces of Framed Maps}
\author{Theodore Spaide}
\date{July 2016}
\maketitle

\begin{abstract}
We examine shifted symplectic and Poisson structures on spaces of framed maps.  We prove some results about shifted Poisson structures analogous to those in \cite{PTVV} and \cite{Calaque} about symplectic structures.  Then, we consider the space \(\Map(X,D,Y)\) of maps from \(X\) to \(Y\) framed along a divisor \(D\).  We give conditions under which this space has a shifted symplectic or Poisson structure.  Classical examples of symplectic and Poisson structures are provided with this theorem.
\end{abstract}

\section*{Introduction}
This paper concerns shifted symplectic structures, a topic which lies in the intersection of derived and symplectic geometry.  Shifted symplectic structures are first described in \cite{PTVV}, in which the authors provide several constructions yielding them and show that they give a framework for explaining several ordinary symplectic structures.

The main motivating result for this paper is the following:
\begin{thm}[\cite{PTVV} Theorem 2.5 and discussion]
Let \(X\) be a smooth, proper, geometrically connected \(d\)-dimensional scheme with a trivialization \(\omega_{X/k}\simeq \calo_X\) of the canonical bundle.  Let \(Y\) be a derived Artin stack with an \(n\)-shifted symplectic structure.  Then the mapping stack \(\Map(X,Y)\) has a natural \((n-d)\)-shifted symplectic structure.
\end{thm}
Furthermore, this construction preserves Lagrangian structures on the target.  That is, if \(Z\to Y\) has a Lagrangian structure, then \(\Map(X,Z)\to\Map(X,Y)\) inherits a natural Lagrangian structure (\cite{Calaque}, Theorem 2.10).

The work of this paper is to extend this result in two directions.  First, rather than maps with Calabi-Yau source, we look at \emph{framed maps} with Fano source.  Second, we extend from looking at symplectic structures to Poisson structures.

For the latter goal, we define an \emph{\(n\)-shifted Poisson structure} on a derived Artin stack \(X\) to be the data of a formal derived stack \(Y\) with \((n+1)\)-shifted symplectic structure, and a map \(X\to Y\) with Lagrangian structure.  This generalizes the idea of ordinary Poisson structures in the sense that ordinary Poisson structures on smooth schemes coincide with \(0\)-shifted quasi-Poisson structures, and all \(n\)-shifted symplectic structures are \(n\)-shifted quasi-Poisson structures.  We give several results about shifted Poisson and coisotropic structures analogous to those concerning shifted symplectic and Lagrangian structures.

For framed maps, we consider a smooth variety \(X\) with divisor \(i:D\to X\), and fix a map \(f:D\to Y\).  We then look at the space of framed maps
\[
\Map(X,D,Y,f)=\HoFib_f(\Map(X,Y)\to\Map(D,Y)),
\]
the moduli parametrizing maps \(g:X\to Y\) with \(g\circ i\simeq f\).

With these definitions in hand, we can state our main result:
\begin{thm}[Theorem \ref{thm:frmap}]
Let \(X\) be a \(d\)-dimensional proper smooth scheme and \(D\) an effective divisor.  Suppose \(E\) is an effective divisor of \(X\) such that \(\tilde{D}=2D+E\) is an anticanonical divisor.  Let \(Y\) be a derived Artin stack such that \(\Map(X,Y)\), \(\Map(\tilde{D},Y)\), \(\Map(D,Y)\), and \(\Map(D+E,Y)\) are themselves derived Artin stacks of locally finite presentation over \(k\).  Fix a base map \(f:D\to Y\).
\begin{enumerate}
\item Suppose \(Y\) is \(n\)-shifted symplectic and the projection \(\Map(D+E,Y)\to\Map(D,Y)\) is etale over \(f\).  Then \(\Map(X,D,Y)\) has an \((n-d)\)-shifted symplectic structure.
\item Suppose \(Y\) is \(n\)-shifted Poisson.  Then \(\Map(X,D,Y)\) has an \((n-d)\)-shifted Poisson structure.
\end{enumerate}
\end{thm}
Furthermore, this preserves Lagrangian or coisotropic structures in the target:
\begin{thm}[Theorem \ref{thm:frLagrmap}]
Let \(X\) be a \(d\)-dimensional proper smooth scheme and \(D\) an effective divisor.  Suppose \(E\) is an effective divisor of \(X\) such that \(\tilde{D}=2D+E\) is anticanonical.  Let \(Y\) be a derived Artin stack such that \(\Map(X,Y)\), \(\Map(\tilde{D},Y)\), \(\Map(D,Y)\), and \(\Map(D+E,Y)\) are themselves derived Artin stacks of locally finite presentation over \(k\).  Fix a base map \(f:D\to Y\).  Let \(W\) be a derived Artin stack and pick a map \(s:W\to Y\).
\begin{enumerate}
\item Suppose \(Y\) is \(n\)-shifted symplectic, that the projection \(\Map(D+E,Y)\to\Map(D,Y)\) is etale over \(f\), and that \(\Map(D+E,W)\to\Map(D,W)\) is etale over any lift \(\tilde{f}\) of \(f\).  Suppose \(s:W\to Y\) has a Lagrangian structure.  Then \(\Map(X,D,W)\to\Map(X,D,Y)\) has a natural Lagrangian structure.
\item Suppose \(Y\) is \(n\)-shifted Poisson, and \(s:W\to Y\) has a coisotropic structure.  Then \(\newline{\Map(X,D,W)\to\Map(X,D,Y)}\) has a natural coisotropic structure.
\end{enumerate}
\end{thm}
Note that even if the target is shifted symplectic, we still might end up with a (non-symplectic) shifted Poisson structure in the end, because we may not have the etaleness property required in part (1) of these theorems.  Thus the Poisson language is necessary for these results.

Finally, we use these results to examine some classical examples of symplectic and Poisson structures, such as spaces of monopoles and instantons.

Further results and applications will be given in \cite{PS}.

Section 1 contains an overview of shifted symplectic structures.  It collects some definitions and results but is not a detailed reference.

Section 2 introduces the subject of shifted Poisson structures.  Results concerning shifted symplectic structures and Lagrangian morphisms are generalized to Poisson structures and coisotropic morphisms.

Section 3 concerns spaces of \emph{framed maps}.  We give conditions under which such spaces have shifted symplectic or Poisson structures.

Section 4 collects a few examples of the results.  We reproduce some classical results concerning framed mapping spaces.

\subsection{Acknowledgments}
This paper is the result of my PhD thesis work.  Thus I would like to give my advisor Tony Pantev an endless amount of gratitude.  I would also like to thank Ludmil Katzarkov and Betrand To\"en for their advice and suggestions.

The author was partially supported by FWF grant P24572 and a Simons Collaboration Grant on Homological Mirror Symmetry.

\section{Shifted Symplectic Structures}
In the following, \(k\) will be the base field, of characteristic \(0\).

Shifted symplectic structures are first defined in \cite{PTVV}.  We will recall some definitions and results.

Let \(X\) be a derived Artin stack.  We can form the de Rham algebra \(\Omega^*_X=\Sym^*_{\calo_X}(\LL_{X}[1])\).  This is a weighted sheaf whose weight \(p\) piece is \(\Omega^p_X=\Sym^p_{\calo_X}(\LL_{X}[1])=\wedge^p\LL_{X}[p]\).

The \emph{space of \(p\)-forms of degree \(n\) on \(X\)} is (the geometric realization of)
\[
\cala^p(X,n)=\tau_{\leq 0}\Hom_{L_{QCoh(X)}}(\calo_X,\wedge^p\LL_X[n]).
\]

We also construct the weighted negative cyclic chain complex \(NC^w\), whose degree \(n\), weight \(p\) part is
\[
NC^n(\Omega_X)(p)=(\bigoplus_{i\geq0}\wedge^{p+i}\LL_X[n-i],d_{\LL_X}+d_{dR}).
\]
The \emph{space of closed \(p\)-forms of degree \(n\)} is
\[
\cala^{p,cl}(X,n)=\tau_{\leq 0}\Hom_{L_{QCoh(X)}}(\calo_X,NC^n(\Omega_X)(p)).
\]
There is a natural ``underlying form'' map \(\cala^{p,cl}(X,n)\to \cala^{p}(X,n)\) corresponding to the projection
\[
\bigoplus_{i\geq0}\wedge^{p+i}\LL_X[n-i]\to \wedge^{p}\LL_X[n].
\]

A \(2\)-form \(\omega:\calo_X\to\wedge^2\LL_X[n]\) of degree \(n\) is \emph{nondegenerate} if the adjoint map \(\TT_X\to\LL_X[n]\) is a quasi-isomorphism.  An \emph{\(n\)-shifted symplectic form} on \(X\) is a closed \(2\)-form whose underlying form is nondegenerate.

\subsection{Lagrangian Structures}
Let \(Y\) be a derived Artin stack with an \(n\)-shifted symplectic form \(\omega\) and let \(f:X\to Y\) be a morphism.  An \emph{isotropic structure} on \(f\) is a homotopy \(h:0\sim f^*\omega\).

An isotropic structure on \(f\) defines a map \(\Theta_h:\TT_f\to\LL_X[n-1]\).  We say \(h\) is \emph{Lagrangian} if \(\Theta_h\) is a quasi-isomorphism of complexes.

For future motivational use, we note that if \(\bullet_n\) denotes the point with trivial \(n\)-shifted symplectic structure, then a Lagrangian structure on \(X\to \bullet_n\) is just an \((n-1)\)-shifted symplectic structure on \(X\).

Lagrangian structures are useful in generating new symplectic spaces:
\begin{thm}[\cite{PTVV}, Theorem 2.9]
Let \(X, L_1, L_2\) be derived Artin stacks, \(\omega\in \Symp(X,n)\) an \(n\)-shifted symplectic structure on \(X\), and \(f_i:L_i\to X\) a morphism with Lagrangian structure \(h_i\) for \(i=1,2\).  Then the product \(L_1\times^h_XL_2\) has a natural \((n-1)\)-shifted symplectic structure, which we denote by \(R(\omega, h_1,h_2)\).
\label{thm:Lagrint}
\end{thm}

\subsection{Symplectic Structures on Mapping Stacks}
Let \(X\) and \(Y\) be derived Artin stacks, with \(Y\) having an \(n\)-shifted symplectic form.  A \emph{\(d\)-orientation} \([X]\) on \(X\) is a ``fundamental class'' \([X]:\RR\Gamma(X,\calo_X)\to k[-d]\) satisfying certain nondegeneracy properties.  Given such an orientation, we can construct a symplectic form on \(\Map(X,Y)\):
\begin{thm}[\cite{PTVV}, Theorem 2.5]  Let \(Y\) be a derived Artin stack, and let \(X\) be an \(\calo\)-compact derived stack with a d-orientation \([X]\).  Assume the derived mapping stack \(\Map(X,Y)\) is a derived Artin stack locally of finite presentation over \(k\).  Then we have a map
\[
\int_{[X]}ev^*(-):\Symp(Y,n)\to\Symp(\Map(X,Y),n-d).
\]
\label{thm:sympmapthm}
\end{thm}

Several examples of orientations are given in \cite{PTVV}, following Theorem 2.5.  One particular example is the case that \(X\) is Calabi-Yau.  If \(X\) has dimension \(d\) and we have an isomorphism \(\omega_X\simeq\calo_X\), then projection of \(\RR\Gamma(X,\calo_X)\) onto the degree \(d\) cohomology \(H^d(X,\calo_X)[-d]\), followed by the isomorphism
\[
H^d(X,\calo_X)\simeq H^d(X,\omega_X)\simeq k
\]
provides a map \([X]:\RR\Gamma(X,\calo_X)\to k[-d]\).

\subsection{Boundary Structures}
The following is due to \cite{Calaque}.  Let \(X\) and \(Y\) be derived Artin stacks.  Assume that \(X\) and \(Y\) are \(\calo\)-compact and \(X\) has a \(d\)-orientation \([X]\), and let \(f:X\to Y\) be a map.  A \emph{boundary structure} on \(f\) is a homotopy \(0\sim f_*[X]\).  This is dual to the notion of isotropic structures, and one can define a nondegeneracy condition dual to the notion of Lagrangian structures.

Then we have:
\begin{thm}[\cite{Calaque}, Theorem 2.9]
Let \(X\) and \(Y\) be as above.  Let \(Z\) be a derived Artin stack with an \(n\)-shifted symplectic structure.  Then a nondegenerate boundary structure on \(f\) yields a Lagrangian structure on \((f\circ-):\Map(X,Z)\to\Map(Y,Z).\)
\label{thm:bdry}
\end{thm}

In particular we are interested in the following case.  Let \(X\) be a geometrically connected smooth proper algebraic variety of dimension \(d+1\), and say it has an anticanonical effective divisor \(D\).  Then \(D\) is a \(d\)-dimensional Calabi-Yau variety by the adjunction formula, and so has a \(d\)-orientation \([D]:\RR\Gamma(D,\calo_D)\to k[-d]\).  Then we have:
\begin{lem} [\cite{Calaque}, Claim 3.3]
There is a canonical nondegenerate boundary structure on \(i:D\to X\).
\end{lem}
And so, using Theorem \ref{thm:bdry}, we get
\begin{cor}
Let \(X\) be a geometrically connected smooth proper algebraic variety of dimension \(d+1\), and let \(D\) be an anticanonical effective divisor.  Let \(Y\) have an \(n\)-shifted symplectic form \(\omega\in\Symp(Y,n)\).  Assume \(\Map(X,Y)\) and \(\Map(D,Y)\) are derived Artin stacks.

Then there exist a natural \((n-d)\)-shifted symplectic form on \(\Map(D,Y)\) and Lagrangian structure on \(\Map(X,Y)\to\Map(D,Y)\).
\label{cor:acanLag}
\end{cor}

\section{Shifted Poisson Structures}
To motivate the definition of shifted Poisson structures, let us first look at ordinary Poisson structures in terms of shifted symplectic structures.  Let \(X\) be a smooth (underived) scheme with Poisson bivector field \(\pi\).  The Poisson structure is equivalently given by the sheaf map \(\pi^\sharp:T^\vee_X\to T_X\).

Here is one way to get maps \(T^\vee_X\to T_X\).  Let \(Y\) be a formal stack with a \(1\)-shifted symplectic structure, and let \(q:X\to Y\) have a Lagrangian structure \(h\).  Then the Lagrangian condition gives a quasi-isomorphism \(\TT_q\simeq T^\vee X\), and composing with \(\TT_q\to T_X\) yields a map \(\pi^\sharp_h:T^\vee X\to TX\).

This construction necessarily yields a Poisson structure, and in fact provides all Poisson structures:
\begin{thm}
Let \(X\) be a smooth scheme.  Then:
\begin{enumerate}
\item Given a Poisson structure \(\pi\) on \(X\), there exist a formal derived stack \(Y\) with \(1\)-shifted symplectic \(\omega\), a map \(q:X\to Y\), and Lagrangian structure \(h\) on \(q\) such that \(\pi^{\sharp}=\pi^\sharp_h\).
\item Let \(Y\) be a formal derived stack with \(1\)-shifted symplectic form \(\omega\), and let \(q:X\to Y\) be a map with Lagrangian structure \(h\).  Then \(\pi^\sharp_h\) is a Poisson structure on \(X\).
\label{thm:Poisclas}
\end{enumerate}
\end{thm}
\begin{proof}
For (1), consider the map \(\pi^\sharp:TX\to T^\vee X\).  This map extends to a map
\[
\wedge^p\pi^\sharp:\wedge^pT^\vee_X\to \wedge^pT_X
\]
such that the square
\[
\xymatrix{
\wedge^pT^\vee_X\ar[r]^{\wedge^p\pi^\sharp}\ar[d]^d&\wedge^pT_X\ar[d]^{[\pi,-]}\\
\wedge^{p+1}T^\vee_X\ar[r]^{\wedge^{p+1}\pi^\sharp}&\wedge^{p+1}T_X
}
\]
commutes, where \([-,-]\) is the Schouten bracket.  To see this, first note that \([\pi,-]\) has square \(0\): for \(a\in\Gamma(U,\wedge^p TX)\), we have \([\pi,[\pi,a]]=\frac{1}{2}[[\pi,\pi],a]\), but \([\pi,\pi]=0\) is exactly equivalent to the Jacobi identity for \(\pi\).  Additionally, \([\pi,-]\) is a derivation: \([\pi,ab]=[\pi,a]b+(-1)^aa[\pi,b]\).  The claim holds almost by definition for \(p=0\): for \(f\in \Gamma(U,\calo_X)\), we have
\[
[\pi,a]=\iota_{df}\pi=\pi^\sharp(df).
\]
Assuming the claim for \(p-1\), note that \(\Gamma(U,\wedge^pT^\vee_X)\) is generated \(k\)-linearly by sections of the form \(fd\alpha\), for \(\alpha\in \Gamma(U,\wedge^{p-1}T^\vee_X)\).  Then
\begin{align*}
\wedge^{p+1}\pi^\sharp(d(fd\alpha))&=\wedge^{p+1}\pi^\sharp(df\wedge d\alpha)\\
&=\pi^\sharp(df)\wedge (\wedge^{p}\pi^\sharp(d\alpha)),
\end{align*}
and
\begin{align*}
[\pi,\wedge^{p}\pi^\sharp(fd\alpha)]&=[\pi,f\wedge^{p}\pi^\sharp(d\alpha)]\\
&=[\pi,f]\wedge[\pi,\wedge^{p-1}\pi^\sharp(\alpha)]+f[\pi,[\pi,\wedge^{p-1}\pi^\sharp(\alpha)]]\\
&=\pi^\sharp(df)\wedge (\wedge^{p}\pi^\sharp(d\alpha)).
\end{align*}
Thus the map \(T^\vee X\to TX\) induces a morphism
\[
\Sym^\bullet(\pi^\sharp[-1]):(\Sym^\bullet(T^\vee X[-1]),d)\to(\Sym^\bullet(TX[-1]),[\pi,-])
\]
of graded mixed cdga.

We can then form the derived quotient \([X/\pi^\sharp]\).  This is a formal stack equipped with a map \(q:X\to[X/\pi^\sharp]\).  It satisfies the universal property that a map \(f:X\to F\) to a formal derived stack \(F\) descends to \(\varphi:[X/\pi^\sharp]\to F\) iff the map \(\Sym^\bullet(\pi^\sharp[1])\) factors through
\[
\psi:(\Sym^\bullet(\LL_{f,big}[-1]),d)\to(\Sym^\bullet(TX[-1]),[\pi,-]),
\]
and a map \(\psi\) of mixed graded cdgas uniquely determines \(\varphi\).

The structure sheaf of this stack is \((\Sym^\bullet_{\calo_X}(TX[-1]),[\pi,-])\).  Its tangent complex is
\[
\TT_{[X/\pi^\sharp]}\simeq\{\xymatrix{T^\vee X\ar[r]^{\pi^\sharp}&TX}\},
\]
with \(TX\) sitting in degree \(0\).  Looking at the \(2\)-forms, we have
\[
\wedge^2\LL_{[X/\pi^\sharp]}\simeq\{\wedge^2T^\vee X\to T^\vee X\otimes TX\to\Sym^2 TX\}.
\]
The degree \(1\) component, \(T^\vee X\otimes TX\cong \Hom(TX,TX)\), contains a canonical section \(\omega=id_{TX}\).  For this to define a \(2\)-form we need \(d\omega=0\).  To see this, note that the image of \(\omega\) via
\[
 T^\vee X\otimes TX\to TX\otimes TX
\]
is the bivector field \(\pi\); that this disappears in \(\Sym^2 TX\) is precisely the fact that \(\pi\) is antisymmetric.  Nondegeneracy is clear, as the map \(\TT_{[X/\pi^\sharp]}\to \LL_{[X/\pi^\sharp]}[1]\) is literally the identity using the above representatives for \(\TT_{[X/\pi^\sharp]}\) and \(\LL_{[X/\pi^\sharp]}\).  For closedness, let
\[
\zeta\in(\calo_{[X/\pi^\sharp]})_1\otimes (\LL_{[X/\pi^\sharp]})_0\cong TX\otimes T^\vee X
\]
again be \(id_{TX}\).  Then \(d_{dR}\zeta=\omega\), so we have \(d_{dR}\omega=0\).  Thus we can take \(0\) as a closedness structure for \(\omega\).  (Note that \(\zeta\) does not define a form on \([X/\pi^\sharp]\), as generally \(d\zeta\neq0\); thus \(\omega\) is not necessarily exact.)

Thus \([X/\pi^\sharp]\) has a canonical \(1\)-shifted symplectic structure.  Looking at \(q:X\to[X/\pi^\sharp]\), we see that \(q^*\omega\) is a form of degree \(1\), so it is zero in \(\wedge^2\LL_X\simeq\wedge^2 T^\vee X\).  Thus \(q\) is isotropic with isotropic structure \(0\).  Further, \(\TT_q\simeq T^\vee X\), and the induced map \(\TT_q\to \LL_X\) is the identity.  So in fact, \(q\) has a Lagrangian structure.  Finally, \(\pi^\sharp_0:T^\vee X\simeq\TT_q\to TX\) is exactly the map \(\pi^\sharp\).

Now consider the case of (2).  The quasi-isomorphism \(\TT_q\to \LL_X\) gives us a map \(\pi^\sharp:T^\vee X\simeq\TT_q\to TX\).  Using the fiber sequence
\[
\xymatrix{T^\vee X\ar[r]^{\pi^\sharp}&TX\ar[r]& q^*\TT_Y},
\]
we have
\[
q^*\TT_Y\simeq\{\xymatrix{T^\vee X\ar[r]^{\pi^\sharp}&TX}\},
\]
with \(TX\) sitting in degree \(0\).  Under this identification, we have
\[
q^*(\wedge^2\LL_Y)\simeq\{\wedge^2T^\vee X\to T^\vee X\otimes TX\to\Sym^2 TX\},
\]
with \(q^*\omega\) corresponding to the identity in \(T^\vee X\otimes TX\).  As before, antisymmetry of \(\pi^\sharp\) is exactly the fact that \(d(q^*\omega)=0\).  In particular, \(\pi^\sharp\) corresponds to a bivector field \(\pi\).  For the Jacobi identity, look at the second infinitesimal neighborhood \(X_{q,2}\) of \(X\) along \(q\).  Its structure sheaf is given by \(\calo_{X_{q,2}}\simeq(\Sym^{\leq2}_{\calo_X}(TX[-1]),[\pi,-])\).  For this to be a dg-algebra, we need \([\pi,[\pi,-]]=\frac{1}{2}[[\pi,\pi],-]=0\) on \(\calo_X\), so \([\pi,\pi]=0\), which is the Jacobi identity.
\end{proof}

To further motivate the definition, we note that the Poisson structure coming from \(q:X\to Y\) depends only on a formal neighborhood \(\hat{Y}_X\) of \(X\) in \(Y\), and that two maps \(X\to Y\) and \(X\to Y'\) give the same Poisson structures if and only if there is a symplectic isomorphism \(\hat{Y}_X\simeq\hat{Y}'_X\) preserving the Lagrangian structures.  Thus we can specify a canonical target \(Y\) by requiring it to be a formal thickening of \(X\).

With this in mind, we define:
\begin{defn}
An \(n\)-shifted Poisson structure on \(X\) is \((Y,\omega,q,h)\), where \(Y\) is a formal derived stack with an \((n+1)\)-shifted symplectic structure \(\omega\in\Symp(Y,n+1)\), \(q:X\to Y\) is a map such that \(X_{red}\to Y_{red}\) is an isomorphism, and \(h\) is a Lagrangian structure on \(q\).
\end{defn}

\begin{rmk}
This definition differs from the one found in \cite{CPTVV}.  However, the two definitions are shown to be equivalent in a forthcoming paper by K. Costello and N. Rozenblyum \cite{CR}.  More specifically, for any reasonable derived stack \(X\), the authors define prestacks of Poisson structures on \(X\) (the \cite{CPTVV} definition) and nil-Lagrangian thickenings of \(X\) (the above definition), and show these are equivalent.
\end{rmk}

We will make frequent implicit use of the following lemma, whose proof is obvious:
\begin{lem}
Let \(X\) be a derived stack and \(Y\) a formal derived stack with an \((n+1)\)-shifted symplectic structure.  Let \(q:X\to Y\) be a map with a Lagrangian structure.  Then \(X\) has an \(n\)-shifted Poisson structure given by \(X\to \hat{Y}_X\).
\end{lem}

We will say \(X\) is Poisson over \(Y\).

\begin{rmk}
As mentioned earlier, if \(\bullet_{(n+1)}\) denotes the point with trivial \((n+1)\)-shifted symplectic structure, then a symplectic structure on a derived Artin stack \(X\) is the same as a Lagrangian structure on the map \(X\to \bullet_{(n+1)}\).  Thus every symplectic structure is naturally Poisson.
\label{rmk:sympois}
\end{rmk}

\begin{rmk}
Let \(X\) be a smooth variety and \(\pi=0\) the zero Poisson structure.  Then the space \([X/\pi^\sharp]\) constructed in the proof of Theorem \ref{thm:Poisclas} is the shifted cotangent space \(T^{\vee}X[1]\).

For nonzero \(\pi\), suppose the (classical) moduli space of symplectic leaves \(Y\) is a derived Deligne-Mumford stack.  Suppose there is a closed \(2\)-form \(\omega\) on \(X\) whose pullback to any symplectic leaf is the form induced by \(\pi\).  Then the map \(X\to T^\vee Y[1]\) with Lagrangian structure \(\omega\) yields the Poisson structure \(\pi\).
\end{rmk}

Now consider a smooth variety \(X\) with Poisson structure \(\pi\), which we consider in terms of the \(1\)-shifted symplectic structure \(\omega\) on some \(Y\) and the Lagrangian structure on \(q:X\to Y\).  Let us now characterize coisotropic subvarieties of \(X\) in terms of the map \(q\) with its Lagrangian structure.

\begin{thm}
Let \(X\) be a smooth variety.  Let \(Y\) be a formal derived stack with \(1\)-shifted symplectic structure \(\omega\), and let \(q:X\to Y\) be a map with Lagrangian structure \(h\).  Let \(\pi\) be the resulting Poisson structure.
\begin{enumerate}
\item Suppose that \(W\) is a coisotropic subvariety of \(X\), and let \(s:W\to X\) be the inclusion.  Then there exists a formal derived stack \(X'\) and maps \(s':W\to X'\), \(q':X'\to Y\), such that \(q'\) has a Lagrangian structure, \(q\circ s=q'\circ s'\), and the induced map \(a:W\to P:=X'\times_YX\) has a Lagrangian structure.
\item Conversely, say \(s:W\to X\) is a subvariety, and suppose there exist a formal derived stack \(X'\), maps \(s':W\to X'\), \(q':X'\to Y\), a Lagrangian structure on \(q'\), a homotopy \(q\circ s\sim q'\circ s'\), and a Lagrangian structure on \(a:W\to P:=X'\times_YX\).  Then \(W\) is coisotropic in \(X\).
\end{enumerate}
\end{thm}
\begin{proof}
For (1), let \(s:W\to X\) be a coisotropic subvariety.  That is, \(W\) is also a smooth variety, and the Poisson structure restricted to the conormal bundle \(N^\vee_{W|X}\to T^\vee X\to TX\) factors through the tangent space \(TW\) of \(W\).  Let the adjoint of \(N^\vee_{W|X}\to TW\) be \(\pi^\sharp_W:T^\vee W\to N_{W|X}\); one can show that the morphism of mixed graded cdgas induced by \(\pi^\sharp\) descends to \(\pi^\sharp_W\), so we have a formal quotient \(X':=[W/\pi^\sharp_W]\), with a projection \(s':W\to[W/\pi^\sharp_W]\).  From the universal property of \([W/\pi^\sharp_W]\) there is a natural map \(q':[W/\pi^\sharp_W]\to Y\) descending from \(W\to X\to Y\).

We can write
\[
\TT_{[W/\pi^\sharp_W]}\simeq\{N^\vee_{W|X}\to TW\},
\]
with \(TW\) in degree \(0\); thus,
\[
\wedge^2\LL_{[W/\pi^\sharp_W]}\simeq\{\wedge^2T^\vee W\to T^\vee W\otimes N_{W|X}\to \Sym^2N_{W|X}\}.
\]
I claim that \((q')^*\omega=0\).  To see this, recall that the pullback of \(\omega\) to \((q')^*\wedge^2\LL_Y\) is the element of \(T^\vee X\otimes TX\) corresponding to the identity on \(TX\).  But \(\Hom(TX,TX)\to\Hom(TW,N_{W|X})\) sends the identity to the composition \(TW\to TX\to N_{W|X}\), which is \(0\) by definition.  Thus \(q\) has isotropic structure \(0\).  Further, we have
\[
\TT_q\simeq\{T^\vee W\to N_{W|X}\},
\]
with \(T^\vee W\) sitting in degree \(0\).  The map \(\TT_q\to \LL^\vee_{[W/\pi^\sharp_W]}\) is clearly an isomorphism, so the isotropic structure on \(q\) is Lagrangian.  Then \(P=[W/\pi^\sharp_W]\times_{Y}X\) is a Lagrangian intersection, so it has a \(0\)-shifted symplectic structure.  One can check that \(\TT_P\) is an extension
\[
0\to TW\to\TT_P\to T^\vee W\to 0.
\]
Let \(a:W\to P\) be map induced by \(s'\) and \(s\); if \(\omega_P\) is the symplectic form on \(P\), then \(a^*\omega_P=0\), so \(a\) is isotropic (with isotropic structure \(0\)).  However, we have \(\TT_a\simeq T^\vee W[-1]\), and \(\TT_a\to \LL_W[-1]\) is the identity.  So in fact \(a:W\to P\) is Lagrangian.

For (2), let \(pr_1:P\to X'\) be the projection.  Then we have an exact sequence
\[
\TT_a\to \TT_s\to a^*\TT_{pr_1}.
\]
Using \(\TT_a\simeq\LL_W[-1]\) from the Lagrangian structure, and \(\TT_{pr_1}\simeq pr_2^*\TT_q\simeq pr_2^*T^\vee X\), we have
\[
T^\vee W[-1]\to\TT_s\to s^*T^\vee X,
\]
so that \(\TT_s\simeq N^\vee_{W|X}\).  Further, the diagram
\[
\xymatrix{
\TT_s\ar[r]\ar[d]& TW\ar[d]\\
s^*\TT_q\ar[r]& s^*TX
}
\]
commutes, that is, the Poisson map \(N^\vee_{W|X}\to T^\vee X\to TX\) factors through \(N^\vee_{W|X}\to TW\).  So \(W\) is coisotropic in the usual sense.
\end{proof}

Once again, the coisotropic structure depends only on a formal neighborhood of \(W\).

This leads us to define coisotropic structures in general:

\begin{defn}
Let \(X\) be a derived Artin stack with \(n\)-shifted Poisson structure given by \(f:X\to Y\).  Let \(W\) be a derived Artin stack with a map \(g:W\to X\).  A \emph{coisotropic structure} on \(g\) consists of \((X', f', \alpha, s, \beta)\), with:
\begin{itemize}
\item \(X'\) a formal derived stack
\item \(f':X'\to Y\)
\item \(\alpha\) a Lagrangian structure on \(f'\)
\item A lifting of \(g\) to \(s:W\to X\times_Y X'\) such that \(W_{red}\to X'_{red}\) is an isomorphism.
\item \(\beta\) a Lagrangian structure on the map \(s\).  (\(Y'\times_YX\) has an \(n\)-shifted symplectic form by Theorem \ref{thm:Lagrint}.)
\end{itemize}
\end{defn}

Again, we will implicitly use the following lemma:
\begin{lem}
Let \(W\) and \(X\) be derived stacks, and \(X'\) and \(Y\) formal derived stacks.  Say \(Y\) has an \((n+1)\)-shifted symplectic structure and say we have maps \(X\to Y\), \(X'\to Y\), and \(W\to X\times_Y X'\) with Lagrangian structures.

Then giving \(X\) the Poisson structure \(X\to \hat{Y}_X\), the map \(W\to X\) has a coisotropic structure given by the obvious maps \(\hat{X}'_W\to \hat{Y}_X\) and \(W\to \hat{X}'_W\times_{\hat{Y}_X}X\), with the natural symplectic and Lagrangian structures.
\end{lem}

We will say \(W\to X\) is coisotropic over \(X'\to Y\).

\begin{rmk}
It is clear from the definition that if \(X\) has an \(n\)-shifted symplectic structure, considered as an \(n\)-shifted Poisson structure via \(X\to\bullet_{(n+1)}\), then any Lagrangian morphism \(Y\to X\) is also coisotropic over \(\bullet_{n}\to\bullet_{(n+1)}\).
\end{rmk}

Now we can begin proving results for Poisson and coisotropic structures analogous to those of section 2.  It is generally easier to first state them in the language of symplectic structures and Lagrangian morphisms.

\begin{lem}
Let \(X_1, X_2, X_3, Y\) be derived Artin stacks, and let \(\omega\in\Symp(Y,n)\) be an \(n\)-shifted symplectic structure on \(Y\).  For \(i=1,2,3\), let \(f_i:X_i\to Y\) be a morphism with Lagrangian structure \(h_i\).  Note that any product \(X_i\times_Y X_j\) has a canonical \((n-1)\)-shifted symplectic structure.  Let \(g_{12}:L_{12}\to X_1\times_Y X_2\) and \(g_{23}:L_{23}\to X_2\times_Y X_3\) be morphisms with Lagrangian structures \(k_{12}\), \(k_{23}\) respectively.

Then \(g_{13}:L_{13}:=L_{12}\times_{X_2}L_{23}\to X_1\times_{Y}X_3\) has a canonical Lagrangian structure.
\label{lem:Lagrcorr}
\end{lem}
\begin{proof}
Let \(\pi_1:L_{12}\to X_1\) and \(\pi_2:L_{12}\to X_2\) be the projections, and let \(\eta_{12}:f_1\circ\pi_1\to f_2\circ\pi_2\) be the natural equivalence of morphisms.  If \(\omega_{12}\in\Symp(X_1\times_YX_2,n-1)\) is the symplectic form given by Theorem \ref{thm:Lagrint}, then \(g_{12}^*\omega_{12}=\pi_1^*h_1+\eta_{12}^*\omega-\pi_2^*h_2\).  Then \(k_{12}\) gives a path from \(0\) to this form.  Similarly if \(\pi_2':L_{23}\to X_2\) and \(\pi_3:L_{23}\to X_3\) are the projections and \(\eta_{23}:f_2\circ\pi_2'\to f_3\circ\pi_3\) the equivalence of morphisms, then \(k_{23}\) is a path from \(0\) to \((\pi_2')^*h_2+\eta_{23}^*\omega-\pi_3^*h_3\) in \(\mathcal{A}^{2,cl}(L_{23},n-1)\).
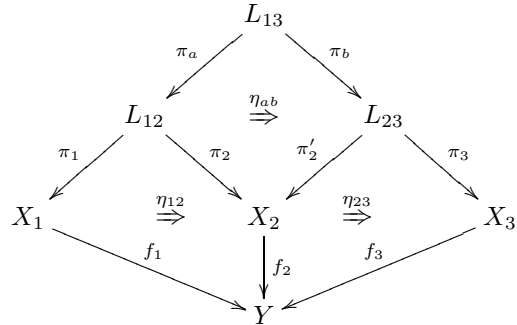
\[
\xymatrix{&&L_{13}\ar[ld]_{\pi_a}\ar[rd]^{\pi_b}\ddtwocell\omit{^\eta_{ab}}\\
&L_{12}\ar[ld]_{\pi_1}\ar[rd]^{\pi_2}\ddtwocell\omit{^<-2>\eta_{12}}&&L_{23}\ar[ld]_{\pi_2'}\ar[rd]^{\pi_3}\ddtwocell\omit{^<2>\eta_{23}}\\
X_1\ar[rrd]^{f_1}&&X_2\ar[d]^{f_2}&&X_3\ar[lld]_{f_3}\\
&&Y&
}
\]
Now let \(\pi_a:L_{13}\to L_{12}\) and \(\pi_b:L_{13}\to L_{23}\) be the projections and \(\eta_{ab}:\pi_2\circ\pi_a\to\pi'_2\circ\pi_b\) the natural equivalence.  Then in \(\mathcal{A}^{2,cl}(L_{13},n-1)\) we have paths
\begin{align*}
\pi_a^*k_{12}&:0\sim \pi_a^*\pi_1^*h_1+\pi_a^*\eta_{12}^*\omega-\pi_a^*\pi_2^*h_2\\
\pi_b^*k_{23}&:0\sim \pi_b^*(\pi_2')^*h_2+\pi_b^*\eta_{23}^*\omega-\pi_b^*\pi_3^*h_3\\
\eta_{ab}^*h_2&:0\sim \pi_a^*\pi_2^*h_2+\eta_{ab}^*f_2^*\omega-\pi_b^*(\pi_2')^*h_2.
\end{align*}
Composing these gives
\begin{equation}
0\sim \pi_a^*\pi_1^*h_1+\eta_{13}^*\omega-\pi_b^*\pi_3^*h_3,
\label{eq:Lagrcorr}
\end{equation}
where
\[
\eta_{13}=(\pi_b\eta_{23})\circ(\eta_{ab}f_2)\circ(\pi_a\eta_{12}):\pi_b\circ\pi_3\circ f_3\to\pi_a\circ\pi_1\circ f_1
\]
is the equivalence.  If \(\omega_{13}\in\Symp(X_1\times_YX_3,n-1)\) is the symplectic form, then (\ref{eq:Lagrcorr}) is exactly the isotropic structure \(0\sim g_{13}^*\omega_{13}\) we need.

For Lagrangianness, we use the diagram
\[
\xymatrix{\TT_{L_{13}}\ar[r]\ar[d]&\TT_{L_{12}}\oplus\TT_{L_{23}}\ar[r]\ar[d]^[right]{\sim}&\TT_{X_2}\ar[d]^[right]{\sim}\ar[r]|-{/}&\\
\LL_{g_{13}}[n-2]\ar[r]&(\LL_{g_{12}}\oplus\LL_{g_{23}})[n-2]\ar[r]&\LL_{f_2}[n-1]\ar[r]|-{/}&
}.
\]
The rows are exact and two of the three vertical maps are quasi-isomorphisms, so the third is as well.
\end{proof}
Restated in Poisson language, the above is a generalization of Theorem \ref{thm:Lagrint}:
\begin{cor}
Let \(X\) have an \(n\)-shifted Poisson structure given by \(f:X\to Y\).  For \(i=1,2\), let \(g_i:X'_i\to X\) be coisotropic over \(h_i:Y'_i\to Y\).  Then \(X_1\times_X X_2\) is \((n-1)\)-shifted Poisson over \(Y_1\times_Y Y_2\).
\label{cor:coisoint}
\end{cor}

Now let us generalize the situation of mapping spaces.  It is relatively clear that Lagrangian structures descend to mapping spaces:
\begin{thm}[\cite{Calaque}, Theorem 2.10]
Let \(X, Y, Z\) be derived Artin stacks and \(f:Y\to Z\) a map.  Assume \(X\) is \(\calo\)-compact with \(d\)-orientation \([X]\).  Assume the stacks \(\Map(X,Y)\) and \(\Map(X,Z)\) are derived Artin stacks locally of finite presentation over \(k\).  Then we have a map
\[
\int_{[X]}ev^*(-):\Lagr(f,\omega)\to\Lagr(f\circ-,\int_{[X]}ev^*(\omega)),
\]
that is, from Lagrangian structures on \(f\) to Lagrangian structures on \((f\circ-)\).
\label{thm:Lagrmap}
\end{thm}
Again, using the language of Poisson structures, we have
\begin{cor}
Let \(Y\) have an \(n\)-shifted Poisson structure over \(Z\).  Let \(X\) be \(\calo\)-compact with \(d\)-orientation \([X]\).  Assume the stacks \(\Map(X,Y)\) and \(\Map(X,Z)\) are derived Artin stacks locally of finite presentation over \(k\).  Then \(\Map(X,Y)\) has an \((n-d)\)-shifted Poisson structure over \(\Map(X,Z)\).
\label{cor:Poismap}
\end{cor}

We also have a variant of Theorem \ref{thm:Lagrmap} to the coisotropic case:
\begin{thm}
Let \(Y\) be \(n\)-shifted Posison over \(Z\), and let \(g:Y'\to Y\) be coisotropic over \(h:Z'\to Z\).  Let \(X\) be \(\calo\)-compact with \(d\)-orientation \([X]\).  Assume the stacks \(\Map(X,Y)\), \(\Map(X,Z)\), \(\Map(X,Y')\), and \(\Map(X,Z')\) are derived Artin stacks locally of finite presentation over \(k\).  Then \(\Map(X,Y')\to\Map(X,Y)\) is coisotropic over \(\Map(X,Z')\to\Map(X,Z)\).
\label{thm:coisomap}
\end{thm}
\begin{proof}
By Theorem \ref{thm:Lagrmap}, the maps \(\Map(X,Z')\to\Map(X,Z)\), \(\Map(X,Y)\to\Map(X,Z)\), and \(\Map(X,Y')\to\Map(X,Z'\times_{Z}Y)\) have natural Lagrangian structures.  But
\[
\Map(X,Z'\times_{Z}Y)\cong\Map(X,Z')\times_{\Map(X,Z)}\Map(X,Y)
\]
as symplectic spaces.
\end{proof}
And similarly of Theorem \ref{thm:bdry}:
\begin{thm}
Let \(Y\) be \(n\)-shifted Poisson given by \(f:Y\to Z\) and Lagrangian structure \(h:0\sim\omega\).  Let \(g:W\to X\) be a map of \(\calo\)-compact derived Artin stacks, and let \([W]\) be a \(d\)-orientation on \(W\) and \(\gamma\) a boundary structure on \(g\).

Then \(\Map(X,Y)\to\Map(W,Y)\times_{\Map(W,Z)}\Map(X,Z)\) has a natural Lagrangian structure.  Equivalently, \(\Map(X,Y)\to \Map(W,Y)\) has a coisotropic structure over \(\Map(X,Z)\to\Map(W,Z)\).
\label{thm:coisobdry}
\end{thm}
\begin{proof}
\(\Map(W,Z)\) has symplectic structure \(\int_{[W]}ev_W^*\omega\).  The Lagrangian structure on \(\Map(W,Y)\to\Map(W,Z)\) is given by
\[
\int_{[W]}ev_W^*h:0\sim\int_{[W]}ev_W^*g^*\omega=(g\circ-)^*\int_{[W]}ev_W^*\omega.
\]
The Lagrangian structure on \(\Map(X,Z)\to\Map(W,Z)\) is given by
\[
\int_{\gamma}ev_X^*\omega:0\sim\int_{f_*[W]}ev_X^*\omega=(-\circ f)^*\int_{[W]}ev_W^*\omega.
\]
Let \(\tilde{\omega}\) be the induced symplectic structure on \(\Map(W,Y)\times_{\Map(W,Z)}\Map(X,Z)\), and let
\[
r:\Map(X,Y)\to \Map(W,Y)\times_{\Map(W,Z)}\Map(X,Z)
\]
be the natural map.  Then
\[
r^*\tilde{\omega}=\int_{f_*[W]}ev_X^*h-\int_{\gamma}ev_W^*g^*\omega,
\]
and the isotropy is given by
\[
\int_\gamma ev_X^*h:0\sim r^*\tilde{\omega}.
\]
For the Lagrangian condition, fix a dga \(A\) and \(\sigma:\Spec A\to\Map(X,Y)\) corresponding to \(\tilde{\sigma}:X\times\Spec A\to Y\).  Let \(\pi_2:X\times\Spec A\to\Spec A\) be the projection.  Then
\[
\sigma^*\TT_r\simeq(\pi_2)_*\HoFib(\tilde{\sigma}^*\TT_g\to(f\times1_{\Spec A})_*(f\times1_{\Spec A})^*\tilde{\sigma}^*\TT_g)
\]
and
\[
\sigma^*\LL_{\Map(X,Y)}\simeq((\pi_2)_*\tilde{\sigma}^*\TT_Y)^\vee=\Hom\left((\pi_2)_*\tilde{\sigma}^*\TT_Y,\calo_{\Spec A}\right).
\]
The map \(\sigma^*\TT_r\to\sigma^*\LL_{\Map(X,Y)}[n-d-1]\) is induced by the maps \(\TT_g\to\LL_Y[n]\), a quasi-isomorphism given by the Lagrangian structure \(h\), and
\[
(\pi_2)_*\HoFib(\tilde{\sigma}^*\LL_Y\to(f\times1_{\Spec A})_*(f\times1_{\Spec A})^*\tilde{\sigma}^*\LL_Y)\to((\pi_2)_*\tilde{\sigma}^*\TT_Y)^\vee[d-1],
\]
a quasi-isomorphism given by the nondegenerate boundary structure.
\end{proof}

Specifically, we want to generalize the case of \ref{cor:acanLag}:
\begin{cor}
Let \(X\) be a geometrically connected smooth proper algebraic variety of dimension \(d+1\), and let \(D\) be an anticanonical effective divisor.  Let \(Y\) have an \(n\)-shifted Poisson structure given by \(Y\to Z\).  Assume \(\Map(X,Y)\), \(\Map(D,Y)\),\(\Map(X,Z)\), and \(\Map(D,Z)\) are derived Artin stacks.

Then there exist a natural \((n-d)\)-shifted Poisson structure on \(\Map(D,Y)\) (over \(\Map(D,Z)\)) and coisotropic structure on \(\Map(X,Y)\to\Map(D,Y)\) (over \(\Map(X,Z)\to\Map(D,Z)\)).
\label{cor:acancois}
\end{cor}
Finally, we need more technical ``Poisson generalizations'' of some results.  The following is a generalization of the first statement of Theorem \ref{thm:coisobdry}:
\begin{cor}
Let \(Y\) be \(n\)-shifted Poisson given by \(f:Y\to Z\).  Let \(C\to Y\) be coisotropic over \(Y'\to Z\).  Let \(g:W\to X\) be a map of \(\calo\)-compact derived Artin stacks, and let \([W]\) be a \(d\)-orientation on \(W\) and \(\gamma\) a boundary structure on \(g\).

Then \(\Map(X,C)\to\Map(W,C)\times_{\Map(W,Y)}\Map(X,Y)\) has a natural coisotropic structure.
\label{cor:coisobdrytwo}
\end{cor}
\begin{proof}
Recall that the Poisson structure on \(\Psi:=\Map(W,C)\times_{\Map(W,Y)}\Map(X,Y)\) is given by
\[
\Psi\to\Xi:=\Map(X,Z)\times_{\Map(W,Z)}\Map(W,Y'),
\]
with a natural Lagrangian structure, as per Corollary \ref{cor:coisoint}.  I claim that\(\Map(X,C)\to\Psi\) is coisotropic over \(\Map(X,Y')\to\Psi\).  First, the Lagrangian structure on \(\Map(X,Y')\to\Psi\) is exactly the one given by Theorem \ref{thm:coisobdry}.  Next, note that \(\Map(X,Y')\times_{\Psi}\Xi\) is exactly the limit of the diagram
\[
\xymatrix{
&&\Map(W,C)\ar[dr]\ar'[d][dd]\\
&\Map(X,Y)\ar[rr]\ar[dd]&&\Map(W,Y)\ar[dd]\\
\Map(X,Y')\ar'[r][rr]\ar[dr]&&\Map(W,Y')\ar[dr]\\
&\Map(X,Z)\ar[rr]&&\Map(W,Z)
},
\]
which we can rewrite as \(\Map(X,Y\times_ZY')\times_{\Map(W,Y\times_ZY')}\Map(W,C)\).  This space has a symplectic form arising from the form on \(Y\times_ZY'\), which will agree with the structure on \(\Map(X,Y')\times_{\Psi}\Xi\) up to sign.  But the fact that
\[
\Map(X,C)\to\Map(X,Y\times_ZY')\times_{\Map(W,Y\times_ZY')}\Map(W,C)
\]
has a Lagrangian structure is precisely Theorem \ref{thm:coisobdry}.
\end{proof}
This is a generalization of Lemma \ref{lem:Lagrcorr}:
\begin{cor}
Let \(X_1, X_2, X_3, Y\) be derived Artin stacks, and let \(Y\) have an \(n\)-shifted Poisson structure given by \(Y\to Z\) for some \((n+1)\)-shifted symplectic \(Z\).  For \(i=1,2,3\), let \(f_i:X_i\to Y\) be a morphism coisotropic over \(Y_i'\to Z\).  Note that any product \(X_i\times_Y X_j\) has a canonical \((n-1)\)-shifted Poisson structure over \(Y_i'\times_Y Y_j'\).  Let \(g_{12}:C_{12}\to X_1\times_Y X_2\) and \(g_{23}:C_{23}\to X_2\times_Y X_3\) be morphisms coisotropic over \(L_{12}\to Y_1'\times_Y Y_2'\) and \(L_{23}\to Y_2'\times_Y Y_3'\), respectively.

Then \(C_{13}:=C_{12}\times_{X_2}C_{23}\to X_1\times_{Y}X_3\) has a canonical coisotropic structure over \(L_{12}\times_{Y'_2}L_{23}\to Y_1'\times_ZY_3'\).
\label{cor:coisocorr}
\end{cor}
\begin{proof}
We need to show that
\[
C_{13}\to T:=(X_1\times_YX_3)\times_{Y_1'\times_ZY_3'}(L_{12}\times_{Y_2'}L_{23})
\]
has a Lagrangian structure.  As in the previous proof, writing \(T\) as a limit gives us
\[
T\cong (X_1\times_{Y_1'}L_{12})\times_{Y\times_ZY_2'}(X_3\times_{Y_3'}L_{23}).
\]
The base is again symplectic, and the maps
\[
X_1\times_{Y_1'}L_{12}\to Y\times_ZY_2'\leftarrow X_3\times_{Y_3'}L_{23}
\]
have Lagrangian structures provided by Lemma \ref{lem:Lagrcorr}.  Thus this expresses \(T\) as a Lagrangian intersection, which again has a symplectic structure that agrees with the original structure on \(T\) up to sign.  Now, \(X_2\to Y\times_ZY_2'\) has a Lagrangian structure by assumption, and further rearrangement gives a Lagrangian structure on
\[
C_{12}\to (X_1\times_Y X_2)\times_{Y_1'\times_ZY_2'} L_{12}\cong (X_1\times_{Y_1'}L_{12})\times_{Y\times_ZY_2'}X_2,
\]
and similarly for \(C_{23}\).

But then we can apply Lemma \ref{lem:Lagrcorr} to get the Lagrangian structure on \(C_{13}\to T\).
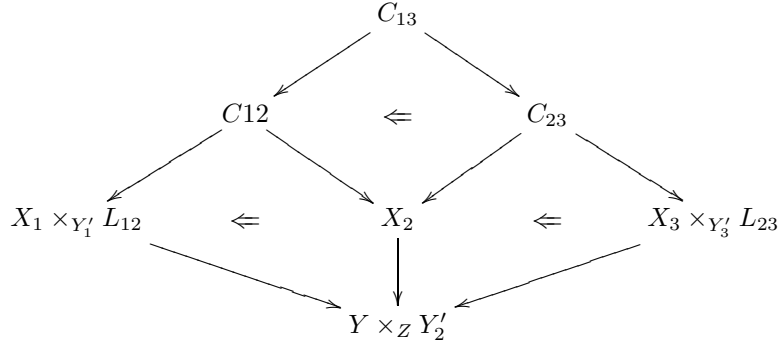
\[
\xymatrix{&&C_{13}\ar[ld]\ar[rd]\ddtwocell\omit\\
&C{12}\ar[ld]\ar[rd]\ddtwocell\omit&&C_{23}\ar[ld]\ar[rd]\ddtwocell\omit\\
X_1\times_{Y_1'}L_{12}\ar[rrd]&&X_2\ar[d]&&X_3\times_{Y_3'}L_{23}\ar[lld]\\
&&Y\times_ZY_2'&
}
\]
\end{proof}

\section{Framed Mapping Spaces}
\begin{defn}
Let \(D, X,\) and \(Y\) be derived Artin stacks, and fix maps \(i:D\to X\) and \({f:D\to Y}\).  We define the \emph{framed mapping space} \(\Map(X,D,Y,f)=\HoFib_f(\Map(X,Y)\to\Map(D,Y))\), the homotopy fiber of \(\Map(X,Y)\) over \(f\in\Map(D,Y)\).  Where \(f\) is understood we will write \(\Map(X,D,Y)\).
\end{defn}

In the following, \(X\) will generally be a smooth scheme and \(i:D\to X\) the inclusion of a divisor; or \(X\) and \(D\) will both be divisors in some smooth scheme.

Fix a cdga \(A\) and consider an \(A\)-point \(\underline{g}:\RSpec A\to \Map(X,D,Y)\) corresponding to \(g:X\times\RSpec A:=X_A\to T\) framed along \(D_A:=D\times\RSpec A\).
What is \(\underline{g}^*\TT_{Map(X,D,Y)}=\TT_{Map(X,D,Y),g}\)?.  We have an exact sequence
\[
\xymatrix{
\TT_{\Map(X,D,Y),g}\ar[r]&\TT_{\Map(X,Y),g}\ar[r]\ar[d]^[right]{\sim}&\TT_{\Map(D,Y),i\circ g}\ar[d]^[right]{\sim}\\
&\RR\Gamma(X_A,g^*\TT_Y)&\RR\Gamma(D_A,i^*g^*\TT_Y)
},
\]
so we can identify \(\TT_{\Map(X,D,Y),g}\simeq\RR\Gamma(X_A,(g^*\TT_Y)_{-D_A})\), where \((g^*\TT_Y)_{-D_A}\) is the subsheaf of \(g^*\TT_Y\) vanishing on \(D_A\).  In our cases we will be able to write \(D=V(a)\) locally, so \((g^*\TT_Y)_{-D_A}\simeq a(g^*\TT_Y)\).\footnote{Here and elsewhere I use \(i:D_A\to X_A\) instead of \(i\times id_{\RSpec A}\) or perhaps \(i_A\); similar simplifications exist later.}

More globally, let \(ev:X\times\Map(X,D,Y)\to Y\) be the evaluation map and
\[
\pi:X\times\Map(X,D,Y)\to \Map(X,D,Y)
\]
the projection.  Then \(\TT_{\Map(X,D,Y)}\simeq\pi_*\left((ev^*\TT_Y)_{-(D\times\Map(X,D,Y))}\right)\).

For \(p\geq 0\) we have a cup product map
\[
\wedge^p\TT_{\Map(X,Y)}\sim\wedge^p(\pi_*ev^*\TT_Y)\to\pi_*\wedge^p(ev^*\TT_Y).
\]
This induces a map
\[
\pi_*(\wedge^pev^*\LL_Y)\to (\pi_*(\wedge^pev^*\TT_Y))^\vee\to \wedge^p\LL_{\Map(X,Y)}.
\]
This map is compatible with the mixed structure on both sides, so descends to the level of negative cyclic complexes:
\[
\pi_*ev^*(NC(Y))\to NC(\Map(X,Y)).
\]

With this in mind, we define a special class of forms on \(\Map(X,Y)\):
\begin{defn}
A \(p\)-form on \(\Map(X,Y)\) (resp. closed \(p\)-form) is \emph{multiplicative} if the corresponding map \(\calo_{\Map(X,Y)}\to\wedge^p\LL_{\Map(X,Y)}\) factors through \(\pi_*(\wedge^pev^*\LL_Y)\) (resp. factors through \(\pi_*ev^*(NC(Y))\)).
\end{defn}
Note that all forms obtained from the \(\int_{[X]}ev^*(-)\) map of Theorem \ref{thm:sympmapthm} are multiplicative.

The importance of multiplicative forms is as follows.  Suppose \(\cale_1,\cale_2\) are two sheaves on \(\Map(X,Y)\) with maps \(\cale\to ev^*\TT_Y\), which are orthogonal in the sense that the multiplication map \(\cale_1\otimes\cale_2\to\wedge^2ev^*\TT_Y\) is \(0\).  Then for any \(2\)-form \(\omega\), we have a pullback
\[
\kappa:\wedge^2\LL_{\Map(X,Y)}\to(\pi_*\cale_1)^\vee\otimes(\pi_*\cale_2)^\vee.
\]
If \(\omega\) is multiplicative, then we can lift the pullback through \(\pi_*(\wedge^2ev^*\LL_Y)\to\pi_*(\cale_1^\vee\otimes\cale_2^\vee)\).  This last map is the \(0\) map, so \(\kappa(\omega)\) will be \(0\).

We want to generalize the case of Theorem \ref{thm:sympmapthm} and Corollary \ref{cor:Poismap} to spaces \(\Map(X,D,Y)\).  The main theorem of this section is
\begin{thm}
Let \(X\) be a \(d\)-dimensional proper smooth scheme and \(D\) an effective divisor.  Suppose \(E\) is an effective divisor of \(X\) such that \(\tilde{D}=2D+E\) is an anticanonical divisor.  Let \(Y\) be a derived Artin stack such that \(\Map(X,Y)\), \(\Map(\tilde{D},Y)\), \(\Map(D,Y)\), and \(\Map(D+E,Y)\) are themselves derived Artin stacks of locally finite presentation over \(k\).  Fix a base map \(f:D\to Y\).
\begin{enumerate}
\item Suppose \(Y\) is \(n\)-shifted symplectic and the projection \(\Map(D+E,Y)\to\Map(D,Y)\) is etale over \(f\).  Then \(\Map(X,D,Y)\) has an \((n-d)\)-shifted symplectic structure.
\item Suppose \(Y\) is \(n\)-shifted Poisson.  Then \(\Map(X,D,Y)\) has an \((n-d)\)-shifted Poisson structure.
\end{enumerate}
\label{thm:frmap}
\end{thm}
\begin{proof}
Consider the fiber diagram
\[
\xymatrix{
\Map(X,D,Y)\ar[r]\ar[d]&\Map(\tilde{D},D,Y)\ar[r]\ar[d]&\bullet\ar[d]\\
\Map(X,Y)\ar[r]&\Map(\tilde{D},Y)\ar[r]&\Map(D,Y)
},
\]
where both squares and the larger rectangle are Cartesian.  Then \(\Map(\tilde{D},Y)\) is \((n-d+1)\)-shifted symplectic (resp. Poisson) by Theorem \ref{thm:sympmapthm} (Corollary \ref{cor:Poismap}), and \(\Map(X,Y)\to\Map(\tilde{D},Y)\) has a canonical Lagrangian structure (coisotropic structure) by Corollary \ref{cor:acanLag} (Corollary \ref{cor:acancois}).  If we can show that \(\Map(\tilde{D},D,Y)\to\Map(\tilde{D},Y)\) has a Lagrangian structure (coisotropic structure) as well, we will be done by Theorem \ref{thm:Lagrint} (Corollary \ref{cor:coisoint}).  We state this as a separate lemma:
\begin{lem}
Let \(X,Y,D,\tilde{D}\) be as in the theorem.  Then
\begin{enumerate}
\item Suppose \(Y\) is \(n\)-shifted symplectic and the projection \(\Map(D+E,Y)\to\Map(D,Y)\) is etale over \(f\).  Then \(\varphi:\Map(\tilde{D},D,Y)\to\Map(\tilde{D},Y)\) has an canonical Lagrangian structure.
\item Suppose \(Y\) is \(n\)-shifted Poisson.  Then \(\Map(\tilde{D},D,Y)\to\Map(\tilde{D},Y)\) has a canonical coisotropic structure.
\end{enumerate}
\label{lem:Lagfib}
\end{lem}
\begin{proof}
Let \(i:D\to\tilde{D}\) be the inclusion, and let \(A\) be a cdga.  Let \(\underline{g}:\RSpec A\to\Map(\tilde{D},D,Y)\) correspond to \(g:\tilde{D}_A\to Y\).  Then
\[
\TT_{\Map(\tilde{D},D,Y),g}\simeq\RR\Gamma(\tilde{D}_A,\HoFib(g^*\TT_Y\to i_*i^*g^*\TT_Y)).
\]
Let us set
\[
(g^*\TT_Y)_{-D_A}=\HoFib(g^*\TT_Y\to i_*i^*g^*\TT_Y)
\]
so that \(\TT_{\Map(\tilde{D},D,Y),g}\simeq \RR\Gamma(\tilde{D}_A,(g^*\TT_Y)_{-D_A})\).  Similarly, for any extension of \(f\) to \(\tilde{f}:D+E\to Y\), we have \(\TT_{\Map(\tilde{D},D+E,Y, \tilde{f}),g}\simeq\RR\Gamma(\tilde{D}_A,(g^*\TT_Y)_{-(D+E)_A})\).

Let us consider (1).  The multiplication
\[
(g^*\TT_Y)_{-D_A}\otimes(g^*\TT_Y)_{-(D+E)_A}\to\wedge^2g^*\TT_Y
\]
is zero; in an affine local patch of \(X\), if \(D=V(a)\) and \(E=V(b)\), then the map on sheaves is
\[
a(g^*\TT_Y)\otimes ab(g^*\TT_Y)\to \wedge^2g^*\TT_Y
\]
and \(a^2b=0\) on \(2D+E\).

Now, the symplectic structure on \(\Map(\tilde{D},Y)\) is a multiplicative form.  Thus, \(\RR\Gamma(\tilde{D},(g^*\TT_Y)_{-D})\) and \(\RR\Gamma(\tilde{D},(g^*\TT_Y)_{-(D+E)})\) are orthogonal in \(\TT_g\Map(\tilde{D},Y)\simeq\RR\Gamma(\tilde{D},g^*\TT_Y)\) under this structure.  Thus the map
\begin{align*}
\RR\Gamma(\tilde{D}_A,(g^*\TT_Y)_{-D_A})&\to\TT_{\Map(\tilde{D},Y),g}\to\LL_{\Map(\tilde{D},Y),g}[n-d+1]\\
&\to\RR\Gamma\left(\tilde{D}_A,(g^*\TT_Y)_{-(D+E)_A}\right)^\vee[n-d+1]
\end{align*}
is \(0\), so
\[
\RR\Gamma(\tilde{D}_A,(g^*\TT_Y)_{-D_A})\to\LL_{\Map(\tilde{D},Y),g}[n-d+1]\simeq\left(\RR\Gamma(\tilde{D}_A,g^*\TT_Y)\right)^\vee[n-d+1]
\]
factors through a map
\begin{equation}
\RR\Gamma(\tilde{D},(g^*\TT_Y)_{-D})\to\left(\RR\Gamma(\tilde{D},g^*\TT_Y\otimes\calo_{D+E})\right)^\vee[n-d+1].
\label{eq:divdual}
\end{equation}
To identify this map, we consider first the diagram:
\[
\xymatrix{
\calo_{X_A}(-(2D+E)_A)\ar[r]\ar[d]&\calo_{X_A}(-(2D+E)_A)\ar[r]\ar[d]&0\ar[d]\\
\calo_{X_A}(-D_A)\ar[r]\ar[d]&\calo_{X_A}\ar[r]\ar[d]&\calo_{D_A}\ar[d]\\
\calo_{(D+E)_A}(-D_A)\ar[r]&\calo_{(2D+E)_A}\ar[r]&\calo_{D}
}.
\]
As all columns and the first two rows are distinguished triangles, so is the last row; tensoring with \(g^*\TT_Y\) on \(\tilde{D}_A\) yields
\begin{align*}
(g^*\TT_Y)_{-D_A}&\simeq \HoFib(g^*\TT_Y\to g^*\TT_Y\otimes\calo_{D_A})\\
&\simeq g^*\TT_Y\otimes\calo_{(D+E)_A}(-D_A)\\
&\simeq g^*\TT_Y\otimes j_*K_{(D+E)_A/A},
\end{align*}
where \(j:(D+E)_A\to\tilde{D}_A\) is the inclusion.

Then the map (\ref{eq:divdual}) can be rewritten as a map
\[
\RR\Gamma((D+E)_A,j^*g^*\TT_Y\otimes K_{(D+E)_A/A})\to\left(\RR\Gamma((D+E)_A,j^*g^*\TT_Y)\right)^\vee[n-d+1].
\]
This is just the quasi-isomorphism
\[
\RR\Gamma((D+E)_A,j^*g^*\TT_Y\otimes K_{(D+E)_A/A})\to\RR\Gamma((D+E)_A,j^*g^*\LL_Y\otimes K_{(D+E)_A/A})[n]
\]
given by the symplectic structure on \(Y\), followed by the Serre-Grothendieck duality quasi-isomorphism
\[
\RR\Gamma((D+E)_A,(j^*g^*\TT_Y)^\vee\otimes K_{(D+E)_A/A})[n]\to\RR\Gamma((D+E)_A,j^*g^*\TT_Y)^\vee[n-d+1].
\]
In particular, (\ref{eq:divdual}) is a quasi-isomorphism.

The etaleness assumption now gives us that
\[
\RR\Gamma(\tilde{D}_A,g^*\TT_Y\otimes\calo_{(D+E)_A})\to\RR\Gamma(\tilde{D}_A,g^*\TT_Y\otimes\calo_{D_A})
\]
is a quasi-isomorphism, and so
\[
\RR\Gamma(\tilde{D}_A,(g^*\TT_Y)_{-(D+E)_A})\to\RR\Gamma(\tilde{D}_A,(g^*\TT_Y)_{-D_A})
\]
is as well.  Thus the map
\[
\wedge^2\RR\Gamma(\tilde{D}_A,(g^*\TT_Y)_{-D_A})\to\wedge^2\RR\Gamma(\tilde{D}_A,g^*\TT_Y)
\]
is \(0\), so \(0\) is an isotropic structure on \(\varphi\).  In addition, the map \(\TT_\varphi\to\LL_{\Map(\tilde{D},D,Y)}[n-d]\) is precisely the map (\ref{eq:divdual}) shifted by \(1\).  This is a quasi-isomorphism, so we have a Lagrangian structure on \(\varphi\).

Now consider case (2).  Suppose \(Y\) has a Poisson structure given by \(p:Y\to Z\), where \(Z\) has an \((n+1)\)-shifted symplectic structure \(\omega\) and \(p\) has a Lagrangian structure \(\gamma\).  Recall that the \((n-d+1)\)-shifted Poisson structure on \(\Map(\tilde{D},Y)\) is given by \(\Map(\tilde{D},Y)\to\Map(\tilde{D},Z)\), with the symplectic and Lagrangian structures induced from \(\omega\) and \(\gamma\).  I claim that
\[
\Map(\tilde{D},D,Y)\to\Map(\tilde{D},Y)
\]
is coisotropic.  The base \(B\) is defined as follows.  First let
\[
\tilde{B}=\Map(D+E,D,Y)\times_{\Map(D+E,D,Z)}\Map(\tilde{D},D,Z).
\]
Then \(\Map(\tilde{D},D,Y)\) has obvious maps to both factors, and thus we get a map \(\Map(\tilde{D},D,Y)\to \tilde{B}\).  Let \(B\) be a formal completion of \(\tilde{B}\) along \(\Map(\tilde{D},D,Y)\).  The required map \(q:B\to \Map(\tilde{D},D,Z)\) comes from the projection \(\tilde{B}\to\Map(\tilde{D},D,Z)\) followed by \(\Map(\tilde{D},D,Z)\to \Map(\tilde{D},Z)\).

First let's find a convenient representation of \(\TT_B\).  Let \(\underline{g}:\RSpec A\to B\) be an \(A\)-point, let \(pr:B\to \Map(\tilde{D},Y)\) be the projection, and let \(pr\circ\underline{g}\) correspond to \(g:\tilde{D}_A\to Y\).  In the diagram
\[
\xymatrix@C=0pt{
\RR\Gamma(\tilde{D}_A,(g^*\TT_p)_{-(D+E)_A})\ar[r]\ar[d]&\RR\Gamma(\tilde{D}_A,(g^*\TT_p)_{-D_A})\ar[rr]\ar[d]&&\RR\Gamma((D+E)_A,(j^*g^*\TT_p)_{-D_A})\ar[d]\\
\RR\Gamma(\tilde{D}_A,(g^*\TT_Y)_{-D_A})\ar[r]\ar[d]&\RR\Gamma(\tilde{D}_A,(g^*\TT_Y)_{-D_A})\oplus\RR\Gamma((D+E)_A,(g^*\TT_Y)_{-D_A})\ar[rr]\ar[d]&&\RR\Gamma((D+E)_A,(g^*\TT_Y)_{-D_A})\ar[d]\\
\TT_{B,\underline{g}}\ar[r]&\RR\Gamma(\tilde{D}_A,(g^*p^*\TT_Z)_{-D_A})\oplus\RR\Gamma((D+E)_A,(g^*\TT_Y)_{-D_A})\ar[rr]&&\RR\Gamma((D+E)_A,(j^*g^*p^*\TT_Z)_{-D_A})
}
\]
the last two columns and all rows are triangles, so the first column is as well.  Thus we have
\begin{equation}
\TT_{B,\underline{g}}\simeq\HoCofib(\RR\Gamma(\tilde{D}_A,(g^*\TT_p)_{-(D+E)_A})\to\RR\Gamma(\tilde{D}_A,(g^*\TT_Y)_{-D_A})).
\label{eq:Btan}
\end{equation}

For the Lagrangian structure on \(q\), let us identify \(q^*\Omega\), where \(\Omega\) is the symplectic structure on \(\Map(\tilde{D},Z)\).  For \(\ell\geq 2\), we have by (\ref{eq:Btan}):
\begin{align*}
\wedge^\ell\LL_{B,\underline{g}}\simeq\{&\wedge^\ell\RR\Gamma(\tilde{D}_A,(g^*\TT_Y)_{-D_A})^\vee\\
&\quad\to\RR\Gamma(\tilde{D}_A,(g^*\TT_p)_{-(D+E)_A})^\vee\otimes\wedge^{\ell-1}\RR\Gamma(\tilde{D}_A,(g^*\TT_Y)_{-D_A})^\vee\\
&\quad\to\cdots\}.
\end{align*}
Now, \(\Omega\) on \(\Map(\tilde{D},Z)\) is multiplicative, and pulling back a multiplicative form to
\[
\Sym^s\RR\Gamma(\tilde{D}_A,(g^*\TT_p)_{-(D+E)_A})^\vee\otimes\wedge^{\ell-s}\RR\Gamma(\tilde{D}_A,(g^*\TT_Y)_{-D_A})^\vee
\]
with \(s>0\) yields \(0\).  The weight \(\ell\) part of \(q^*\Omega\) corresponds to a map \(k\to(\wedge^\ell\LL_B)_g\), which in turn decomposes to a nonzero map \(k\to \wedge^\ell\RR\Gamma(\tilde{D}_A,(g^*\TT_Y)_{-D_A})^\vee\) and a \(0\) map to all later terms in the sequence.  A homotopy from this to \(0\) is given by restricting the isotropic structure \(\int_{[\tilde{D}]}ev^*\gamma\) from \(\wedge^\ell\RR\Gamma(\tilde{D}_A,g^*\TT_Y)^\vee\) to \(\wedge^\ell\RR\Gamma(\tilde{D}_A,(g^*\TT_Y)_{-D})^\vee\), and taking the \(0\) homotopy on all later terms.  This has the following interpretation: the pullback of \(\int_{[\tilde{D}]}ev^*\gamma\) to \(\Map(\tilde{D},D,Y)\) descends to \(B\), which is our isotropic structure on \(q\).

For the Lagrangian condition, using (\ref{eq:Btan}), consider the diagram
\[
\xymatrix{\RR\Gamma(\tilde{D}_A,(g^*\TT_p)\otimes\calo_{(D+E)_A})[-1]\ar[r]\ar[d]&\RR\Gamma(\tilde{D}_A,(g^*\TT_p)_{-(D+E)_A})\ar[r]\ar[d]&\RR\Gamma(\tilde{D}_A,g^*\TT_p)\ar[d]\\
\RR\Gamma(\tilde{D}_A,(g^*\TT_Y)\otimes\calo_{D_A})[-1]\ar[r]\ar[d]&\RR\Gamma(\tilde{D}_A,(g^*\TT_Y)_{-D_A})\ar[r]\ar[d]&\RR\Gamma(\tilde{D}_A,g^*\TT_Y)\ar[d]\\
\TT_{q,\underline{g}}\ar[r]&\TT_{B,\underline{g}}\ar[r]&\RR\Gamma(\tilde{D}_A,g^*p^*\TT_Z)}.
\]
The second two columns and all rows are triangles, so the first column is too.  Thus we have
\begin{equation}
\TT_{q,\underline{g}}\simeq\HoCofib(\RR\Gamma(\tilde{D}_A,(g^*\TT_p)\otimes\calo_{(D+E)_A})[-1]\to\RR\Gamma(\tilde{D}_A,(g^*\TT_Y)\otimes\calo_{D_A})[-1]).
\label{eq:qtan}
\end{equation}
Similarly to the map (\ref{eq:divdual}) above, we obtain a quasi-isomorphism
\[
\RR\Gamma(\tilde{D}_A,(g^*\TT_p)_{-D_A})\to\left(\RR\Gamma(\tilde{D}_A,g^*\TT_Y\otimes\calo_{(D+E)_A})\right)^\vee[n-d+1],
\]
namely
\begin{align*}
\RR\Gamma(\tilde{D}_A,(g^*\TT_p)_{-D_A})&\simeq\RR\Gamma((D+E)_A,j^*g^*\TT_Y\otimes K_{(D+E)_A/\Spec A})\\
&\to\RR\Gamma((D+E)_A,j^*g^*\LL_Y\otimes K_{(D+E)_A/\Spec A})[n],
\end{align*}
where \(\TT_p\to\LL_Y[n]\) is a quasi-isomorphism coming from the Lagrangian structure on \(p\), followed by the Serre-Grothendieck duality quasi-isomorphism
\[
\RR\Gamma((D+E)_A,(j^*g^*\TT_Y)^\vee\otimes K_{(D+E)_A/\Spec A})[n]\to\RR\Gamma((D+E)_A,j^*g^*\TT_Y)^\vee[n-d+1].
\]
Similarly, there is a natural quasi-isomorphism
\[
\RR\Gamma(\tilde{D}_A,(g^*\TT_p)_{-(D+E)_A})\to\left(\RR\Gamma(\tilde{D}_A,g^*\TT_Y\otimes\calo_{D_A})\right)^\vee[n-d+1].
\]
Then the map \(\TT_q\to\LL_B[n-d]\) is given by the diagram
\[
\xymatrix{
\TT_{q,\underline{g}}\ar[r]\ar[d]&\LL_{B,\underline{g}}[n-d+1]\ar[d]\\
\RR\Gamma(\tilde{D}_A,(g^*\TT_p)\otimes\calo_{(D+E)_A})\ar[r]^-{\sim}\ar[d]&\RR\Gamma(\tilde{D}_A,(g^*\TT_Y)_{-D_A})^\vee[n-d+1]\ar[d]\\
\RR\Gamma(\tilde{D}_A,(g^*\TT_Y)\otimes\calo_{D_A})\ar[r]^-{\sim}&\RR\Gamma(\tilde{D}_A,(g^*\TT_p)_{-(D+E)_A})^\vee[n-d+1]\\
};
\]
the columns are triangles by (\ref{eq:Btan}) and (\ref{eq:qtan}).  Then this map is a quasi-isomorphism, so the isotropic structure is Lagrangian.

Let \(Q=\Map(\tilde{D},Y)\times_{\Map(\tilde{D},Z)}B\) be the product, and \(r:\Map(\tilde{D},D,Y)\to Q\) the natural map.  For any \(A\)-point \(\underline{g}:\RSpec A\to \Map(\tilde{D},D,Y)\) corresponding to \(g:\tilde{D}_A\to Y\), consider the diagram
\[
\xymatrix{
\RR\Gamma(\tilde{D}_A,(g^*\TT_p)\otimes\calo_{(D+E)_A})[-1]\ar[r]\ar[d]&\RR\Gamma(\tilde{D}_A,(g^*\TT_p)_{-(D+E)_A})\ar[r]\ar[d]&\RR\Gamma(\tilde{D}_A,g^*\TT_p)\ar[d]\\
\RR\Gamma(\tilde{D}_A,(g^*\TT_Y)_{-D_A})\ar[r]\ar[d]&\RR\Gamma(\tilde{D}_A,g^*\TT_Y)\oplus\RR\Gamma(\tilde{D}_A,(g^*\TT_Y)_{-D_A})\ar[r]\ar[d]&\RR\Gamma(\tilde{D}_A,g^*\TT_Y)\ar[d]\\
(r^*\TT_Q)_g\ar[r]&\RR\Gamma(\tilde{D}_A,g^*\TT_Y)\oplus(r^*\pi_2^*\TT_B)\ar[r]&\RR\Gamma(\tilde{D}_A,g^*\TT_Z);
}
\]
again, everything but the first column is a triangle, so the first column is too.

Then
\[
(r^*\TT_Q)_g\simeq\HoCofib(\RR\Gamma(\tilde{D}_A,(g^*\TT_p)\otimes\calo_{(D+E)_A})[-1]\to\RR\Gamma(\tilde{D}_A,(g^*\TT_Y)_{-D_A})).
\]
The map
\[
\TT_{\Map(\tilde{D},D,Y),g}\simeq\RR\Gamma(\tilde{D}_A,(g^*\TT_Y)_{-D_A})\to(r^*\TT_Q)_g
\]
is precisely the structure morphism of the above cofiber.  Letting \(\omega_Q\) be the symplectic structure on \(Q\), we get \(r^*\omega_Q=0\), so \(r\) has \(0\) as isotropic structure.  It is easy to check that \(\TT_{r,g}\simeq\RR\Gamma(\tilde{D}_A,(g^*\TT_p)\otimes\calo_{(D+E)_A})[-2]\), and the map \(\TT_{r,g}\to\LL_{\Map(\tilde{D},D,Y),g}[n-d-1]\) is the quasi-isomorphism
\[
\RR\Gamma(\tilde{D}_A,(g^*\TT_p)\otimes\calo_{(D+E)_A})\to\RR\Gamma(\tilde{D}_A,(g^*\TT_Y)_{-D_A})^\vee[n-d+1]
\]
shifted by \(-2\).
\end{proof}
\end{proof}

Analogously to Theorems \ref{thm:Lagrmap} and \ref{thm:coisomap}, we have:
\begin{thm}
Let \(X\) be a \(d\)-dimensional proper smooth scheme and \(D\) an effective divisor.  Suppose \(E\) is an effective divisor of \(X\) such that \(\tilde{D}=2D+E\) is anticanonical.  Let \(Y\) be a derived Artin stack such that \(\Map(X,Y)\), \(\Map(\tilde{D},Y)\), \(\Map(D,Y)\), and \(\Map(D+E,Y)\) are themselves derived Artin stacks of locally finite presentation over \(k\).  Fix a base map \(f:D\to Y\).  Let \(W\) be a derived Artin stack and pick a map \(s:W\to Y\).
\begin{enumerate}
\item Suppose \(Y\) is \(n\)-shifted symplectic, that the projection \(\Map(D+E,Y)\to\Map(D,Y)\) is etale over \(f\), and that \(\Map(D+E,W)\to\Map(D,W)\) is etale over any lift \(\tilde{f}\) of \(f\).  Suppose \(s:W\to Y\) has a Lagrangian structure.  Then \(\Map(X,D,W)\to\Map(X,D,Y)\) has a natural Lagrangian structure.
\item Suppose \(Y\) is \(n\)-shifted Poisson, and \(s:W\to Y\) has a coisotropic structure.  Then \(\newline{\Map(X,D,W)\to\Map(X,D,Y)}\) has a natural coisotropic structure.
\end{enumerate}
\label{thm:frLagrmap}
\end{thm}
\begin{proof}
Similarly to the previous theorem, we use the fiber diagram
\[
\xymatrix{
\Map(X,D,W)\ar[rr]\ar[dr]\ar[dd]&&\Map(\tilde{D},D,W)\ar[dr]\ar'[d][dd]\\
&\Map(X,D,Y)\ar[rr]\ar[dd]&&\Map(\tilde{D},D,Y)\ar[dd]^{\theta}\\
\Map(X,W)\ar'[r][rr]\ar[dr]&&\Map(\tilde{D},W)\ar[dr]^{\eta}\\
&\Map(X,Y)\ar[rr]^{\zeta}&&\Map(\tilde{D},Y)
}.
\]
The front and back faces are Cartesian squares.  We have a Lagrangian (resp. coisotropic) structure on \(\zeta\) by Corollary \ref{cor:acanLag} (Corollary \ref{cor:acancois}), on \(\eta\) by Theorem \ref{thm:Lagrint} (Corollary \ref{cor:coisoint}), and on \(\theta\) by Lemma \ref{lem:Lagfib}.  We also have a Lagrangian (coisotropic) structure on
\[
\Map(X,W)\to\Map(X,Y)\times_{\Map(\tilde{D},Y)}\Map(\tilde{D},W)
\]
by Theorem \ref{thm:coisobdry} (Corollary \ref{cor:coisobdrytwo}).  If we show that
\[
\Map(\tilde{D},D,W)\to\Map(\tilde{D},D,Y)\times_{\Map(\tilde{D},Y)}\Map(\tilde{D},W)
\] has a Lagrangian (coisotropic) structure, then we will be done by Lemma \ref{lem:Lagrcorr} (Corolllary \ref{cor:coisocorr}). 
 As before, we put this in a separate lemma:
\begin{lem}
Let \(X,Y,W,D,\tilde{D}\) be as in the theorem.  Then
\begin{enumerate}
\item Suppose \(Y\) is \(n\)-shifted symplectic, that the projection \(\Map(D+E,Y)\to\Map(D,Y)\) is etale over \(f\), and that \(\Map(D+E,W)\to\Map(D,W)\) is etale over any lift \(\tilde{f}\) of \(f\).  Suppose \(s:W\to Y\) has a Lagrangian structure.  Then \(r:\Map(\tilde{D},D,W)\to\Map(\tilde{D},D,Y)\times_{\Map(\tilde{D},Y)}\Map(\tilde{D},W)\) has a canonical Lagrangian structure.
\item Suppose \(Y\) is \(n\)-shifted Poisson and that \(s:W\to Y\) has a coisotropic structure.  Then \(r:\Map(\tilde{D},D,W)\to\Map(\tilde{D},D,Y)\times_{\Map(\tilde{D},Y)}\Map(\tilde{D},W)\) has a canonical coisotropic structure.
\end{enumerate}
\end{lem}
\begin{proof}
For (1), let \(\gamma\) be the Lagrangian structure on \(s\).  If \(\Omega\) is the induced symplectic structure on \(\Map(\tilde{D},D,Y)\times_{\Map(\tilde{D},Y)}\Map(\tilde{D},W)\), then we have \(r^*\Omega=-\int_{[\tilde{D}]}ev^*\gamma\).  But this is a multiplicative form, so is already \(0\) on
\[
\wedge^\ell\TT_{\Map(\tilde{D},D,W),g}\simeq\wedge^\ell\RR\Gamma(\tilde{D}_A,(g^*\TT_W)_{-D_A})\simeq\wedge^\ell\RR\Gamma(\tilde{D}_A,(g^*\TT_W)_{-(D+E)_A}),\quad (\ell\geq 2)
\]
for any cdga \(A\) and \(g:\tilde{D}\to W\) framed along \(D_A\); here the second quasi-isomorphism comes from the etaleness condition.  Thus \(0\) is an isotropic structure.  For the Lagrangian condition, we have
\[
\TT_{r,g}\simeq\RR\Gamma(\tilde{D}_A,g^*\TT_s\otimes\calo_{D_A})\simeq\RR\Gamma(\tilde{D}_A,g^*\TT_s\otimes\calo_{(D+E)_A}),
\]
and the map \(\TT_r\to\LL_{\Map(\tilde{D},D,W)}[n-d]\) is the quasi-isomorphism
\[
\RR\Gamma((D+E)_A,g^*\TT_s)\to\RR\Gamma((D+E)_A,g^*\LL_W)[n-1]
\]
from the Lagrangian condition on \(s\), followed by the Serre-Grothendieck quasi-isomorphism
\begin{align*}
\RR\Gamma((D+E)_A,g^*\LL_W)&\to\RR\Gamma((D+E)_A,g^*\TT_W\otimes K_{(D+E)_A/A})^\vee[1-d]\\
&\simeq\RR\Gamma(\tilde{D}_A,(g^*\TT_W)_{-D_A})^\vee[1-d],
\end{align*}
as in the proof of the previous theorem.

For (2), let the Poisson structure on \(Y\) be given by \(p:Y\to Z\) with Lagrangian structure \(\gamma\), and the coisotropic structure on \(W\to Y\) given by \(u:W\to Y\times_ZY'\) with Lagrangian structure \(\epsilon\), where \(p':Y'\to Z\) has Lagrangian structure \(\gamma'\).

Let
\[
\tilde{B}=\Map(D+E,D,Y)\times_{\Map(D+E,D,Z)}\Map(\tilde{D},D,Z)
\]
and \(B\) a formal neighborhood of \(\Map(\tilde{D},D,Y)\) in \(\tilde{B}\).  Recall that the coisotropic structure on \(\Map(\tilde{D},D,Y)\to\Map(\tilde{D},Y)\) came from the map \(B\to\Map(\tilde{D},Z)\).  In our present case, the Poisson structure on \(\Psi:=\Map(\tilde{D},D,Y)\times_{\Map(\tilde{D},Y)}\Map(\tilde{D},W)\) comes from
\[
\Psi\to \Xi:=B\times_{\Map(\tilde{D},Z)}\Map(\tilde{D},Y').
\]
Let
\[
\tilde{B}'=\Map(D+E,D,W)\times_{\Map(D+E,D,Y')}\Map(\tilde{D},D,Y')
\]
and \(B'\) a formal neighborhood of \(\Map(\tilde{D},D,W)\) in \(\tilde{B}'\).  The maps
\[
\tilde{B}'\to\Map(\tilde{D},D,Y')\to\Map(\tilde{D},Y')
\]
and \(\tilde{B}'\to\tilde{B}\) give a map \(q':B'\to \Xi\).  As in the previous theorem, one can give an isotropic structure on \(q'\) by pulling back the isotropic structure on \(\Map(\tilde{D},W)\to\Map(\tilde{D},Y\times_ZY')\) to \(\Map(\tilde{D},D,W)\), then descending to \(B'\).  For Lagrangianness, fix an \(A\)-point \(\underline{g}:A\to B'\); letting \(pr:B'\to \Map(\tilde{D},D,W)\) be the projection, \(pr\circ\underline{g}\) will correspond to a map \(g:\tilde{D}_A\to W\) framed along \(D_A\).  As in the previous theorem, one shows the columns of the following diagram are distinguished triangles, and the second and third rows are quasi-isomorphisms:
\[
\xymatrix{
\TT_{q',\underline{g}}\ar[r]\ar[d]&\LL_{B',\underline{g}}[n-d+2]\ar[d]\\
\RR\Gamma(\tilde{D}_A,(g^*\TT_u)\otimes\calo_{(D+E)_A})\ar[r]^-{\sim}\ar[d]&\RR\Gamma(\tilde{D}_A,(g^*\TT_W)_{-D_A})^\vee[n-d+2]\ar[d]\\
\RR\Gamma(\tilde{D}_A,(g^*\TT_s)\otimes\calo_{D_A})\ar[r]^-{\sim}&\RR\Gamma(\tilde{D}_A,(g^*\TT_{s'})_{-(D+E)_A})^\vee[n-d+2]\\
}.
\]
Let \(\rho:\Map(\tilde{D},D,W)\to R:=B'\times_{\Xi}\Psi\), and let \(\Omega_R\) be the induced structure on \(R\); then \(\rho^*\Omega_R\) is already \(0\), so \(0\) is an isotropic structure.  For the Lagrangian condition, fix \(\underline{g}:\RSpec A\to \Map(\tilde{D},D,W)\) corresponding to \(g:\tilde{D}_A\to W\) framed along \(D_A\).  The map \(\TT_{\rho,g}\to \LL_{\Map(\tilde{D},D,W),g}[n-d-1]\) is just
\[
\RR\Gamma(\tilde{D}_A,g^*\TT_u\otimes\calo_{(D+E)_A})[-1]\to\RR\Gamma(\tilde{D}_A,(g^*\TT_W)_{-D_A})[n-d-1],
\]
analogous to previous maps.  This gives the coisotropic structure on \(\Map(\tilde{D},D,W)\to \Psi\) we needed.
\end{proof}
\end{proof}
\section{Examples}
\subsection{Framed Vector Bundles on Surfaces}
As an application of these theorems, let \(X\) be a smooth surface with effective anticanonical bundle, and take effective divisors \(D\) and \(E\) such that \(2D+E\) is anticanonical.  Let \(G\) be a reductive group.  Choose a map \(D\to BG\), that is, a \(G\)-bundle \(\mathcal{G}\to D\).  The space \(\Map(X,D,BG,\mathcal{G})\) has, by Theorem \ref{thm:frmap}, a \(0\)-shifted Poisson structure.  This structure will be symplectic if \(\Map(D+E,BG)\to\Map(D,BG)\) is etale over \(\mathcal{G}\).  That is, if for every extension \(\widetilde{\mathcal{G}}\to D+E\) of \(\mathcal{G}\), the map
\[
H^*(D+E,ad(\widetilde{\mathcal{G}}))\to H^*(D,ad(\mathcal{G}))
\]
is an isomorphism in all degrees.  Assuming the moduli space is a smooth variety (or looking at a semistable locus), this will be an ordinary Poisson or symplectic structure.  Taking \(\zeta\in H^0(X,E)\) to be a section vanishing on \(E\), this is precisely Theorem 4.3 of \cite{Bottacin}.

In particular, let us consider the case where \(X=\PP^2\), \(D=E\) is a line \(L\), \(G=SL_n\), and \(\mathcal{G}\) is the trivial bundle.  The space \(\Map(\PP^2,L,BSL_n,\mathcal{G})\) may be identified with the framed \(SU(n)\)-instantons on \(S^4\) (\cite{Donaldson}).  In this case, the only extension of \(\mathcal{G}\) to \(2L\) is the trivial bundle again, and the failure of
\[
H^*(2L,\mathfrak{sl}_n\otimes\calo_{2L})\to H^*(L,\mathfrak{sl}_n\otimes\calo_L)
\]
to be an isomorphism is given by
\[
H^*(L,\mathfrak{sl}_n\otimes\calo_{L}(-1))=0,
\]
so we have a symplectic structure.

\subsection{Monopoles}
Let \(G\) be a semisimple complex Lie group and \(B\) a Borel subgroup.  Let \(Y=G/B\) be the complete flag variety of \(G\).  Fix a point \(p\in\PP^1\).  The space \(\Map(\PP^1,p,Y)\) is the space of framed \(G\)-monopoles on \(\RR^3\) \cite{Jarvis}.  In \cite{FKMM} the authors show that this space has a symplectic structure.  More generally, let \(P\) be a parabolic subgroup of \(G\) and \(Y=G/P\) the partial flag variety; then it is shown that \(\Map(\PP^1,p,Y)\) has a Poisson structure.  Here I will show that the Poisson and symplectic structures arise from the machinery of shifted structures on framed mapping spaces.

The obvious strategy to find a \(1\)-shifted symplectic or Poisson structure on \(Y=G/B\) or \(G/P\) and use Theorem \ref{thm:frmap} to induce a structure on \(\Map(\PP^1,p,Y)\).  This is impossible in the \(P=B\) case, as we would expect a \(1\)-shifted symplectic structure on \(G/B\), which would yield a quasi-isomorphism between \(\frg/\frp\) and \(\frr[1]\).  More generally, it is not too hard to show that whatever Poisson structure we put on \(G/P\), we will end up with the \(0\) Poisson structure on \(\Map(\PP^1,p,G/P)\).

Instead, we note that \(G/P\) is already related to an existing shifted symplectic stack, \(BG\), via the fiber diagram
\begin{equation}
\xymatrix{G/P\ar[r]\ar[d]&BP\ar[d]\\
\bullet\ar[r]&BG
}.
\label{flagsquare}
\end{equation}
Recall that the symplectic structure on \(BG\) is given by a \(G\)-invariant nondegenerate symmetric quadratic form on \(\frg\).  Fix such a form \(\omega\).

Choose an opposite parabolic \(P^-\) so that \(P\cap P^-=L\) is a Levi subgroup of \(G\).  Letting \(\frl=\operatorname{Lie}(L)\), we can then write \(\frg=\frr^-\oplus\frl\oplus\frr\).  Since \(\frl\) is orthogonal to \(\frr\) and \(\frr^-\), \(\omega\) descends to \(\frl\) and is also \(L\)-invariant and nondegenerate.  Thus, \(BL\) has a symplectic structure \(\omega_L\) induced from \(BG\).  Recall that the identification \(L\cong P/\operatorname{rad}(P)\) gives us a map \(P\to L\).
\begin{lem}
The map \(\iota:BP\to BG\times BL\) has a Lagrangian structure given by \(0\).  Thus \(BP\to BG\) has a coisotropic structure.
\end{lem}
\begin{proof}
The claim that \(0\) is an isotropic structure reduces to the claim that \(\iota:\frp\to\frg\oplus \frl\) is isotropic in the usual sense with respect to \(\omega-\omega_L\).  Write \(\frp=\frl\oplus\frr\) and recall that \(\frr\) is orthogonal to itself and \(\frl\).  Then
\[
(\omega-\omega_L)(\iota_*(\ell,r),\iota_*(\ell',r'))=\omega(\ell,\ell')-\omega(\ell,\ell')=0,
\]
so we have isotropy.

For the Lagrangian condition, recall that \(\omega\) pairs up \(\frr\) nondegenerately with \(\frr^-\).  Let \(\newline{\Delta,\overline{\Delta}:\frl\to\frg\oplus\frl}\) be the diagonal and antidiagonal maps, and note that \(\omega-\omega_L\) pairs up \(\Delta(\frl)\) and \(\overline{\Delta}(\frl)\) nondegenerately.  Then the map \(\TT_\iota\to\LL_{BP}[1]\) is just the adjoint \(\overline{\Delta}(\frl)\oplus\frr^-\to(\Delta(\frl)\oplus\frr)^\vee\).
\end{proof}

So \(BP\to BG\) has a coisotropic structure, and if \(\bullet\to BG\) had one too, we would get a \(1\)-shifted Poisson structure on \(G/P\) by Corollary \ref{cor:coisoint}.  As mentioned, there is no decent shifted Poisson structure on \(G/P\) and it is also easy to check that \(\bullet\to BG\) has no coisotropic structure.  Instead, we apply the functor \(\Map(\PP^1,p,-)\) to (\ref{flagsquare}) to get
\[
\xymatrix{\Map(\PP^1,p,G/P)\ar[r]\ar[d]&\Map(\PP^1,p,BP)\ar[d]\\
\Map(\PP^1,p,\bullet)\ar[r]&\Map(\PP^1,p,BG)
},
\]
and note that now \(\Map(\PP^1,p,BG)\) is \(1\)-symplectic and \(\Map(\PP^1,p,BP)\to\Map(\PP^1,p,BG)\) is coisotropic by Theorem \ref{thm:frLagrmap}.  Let \(G_{\PP^1}\) denote the trivial \(G\)-bundle on \(\PP^1\), framed at \(p\).  Then the map
\[
\bullet=\Map(\PP^1,p,\bullet)\to \Map(\PP^1,p,BG)
\]
is just the point \(G_{\PP^1}\).  And this \emph{is} coisotropic, as
\[
\TT_{G_{\PP^1}}\Map(\PP^1,p,BG)=\mathcal{E}xt^*(\calo_{\PP^1}\otimes\frg,\calo_{\PP^1}\otimes\frg)(-1)[1]\simeq0.
\]
Then the map \(\bullet\to\Map(\PP^1,p,BG)\) is trivially Lagrangian, hence coisotropic.  Then Corollary \ref{cor:coisoint} gives a \(0\)-shifted Poisson structure on \(\Map(\PP^1,p,G/P)\).

To be a little more specific, the maps are coisotropic over \(\Map(\PP^1,p,\overline{BL})\to\bullet\) and \(\bullet\to\bullet\) respectively, so the Poisson structure comes from a map \(\Map(\PP^1,p,G/P)\to\Map(\PP^1,p,BL)\).  In particular, in the case \(P\) is a Borel, \(L\) is a torus, and so \(\Map(\PP^1,p,BL)=\prod_{i=1}^r\Map(\PP^1,p,B\mathbb{G}_m)\), where \(r\) is the rank of \(G\).  But \(\Map(\PP^1,p,B\mathbb{G}_m)\cong\mathbb{Z}\), so \(\Map(\PP^1,p,BL)\) is a disjoint union of points \(\bullet_1\) with the trivial \(1\)-symplectic structure.  Thus, \(\Map(\PP^1,p,G/P)\) is in fact symplectic.

\subsection{Projectivized Cotangent Bundles}
Let \(W\) be a smooth compact projective \(d\)-dimensional variety and consider the projective bundle \(X=\PP(\calo_W\oplus T^\vee W)\xrightarrow{\pi}W\).  Let \(\calo_Y(1)\) be the relative hyperplane bundle such that \(\pi_*\calo_X(1)\simeq \calo_W\oplus T^\vee W\), and let \(D\cong\PP(T^\vee W\) be the divisor at infinity.

Then we have \(\calo_X(1)\cong\calo_X(D)\), and \(K_X\cong \calo_X(-(d+1)D)\), so \(X\) is Fano.  Taking \(E=(d-1)D\), let us apply Theorem \ref{thm:frmap} to the case where \(Y=BGL_n\).  Fixing a vector bundle \(F\) on \(D\), the moduli space \(\Map(X,D,BGL_n,F)\) consists of vector bundles \(\tilde{F}\) on \(X\) with \(\tilde{F}|_D\simeq F\).  Theorem \ref{thm:frmap} tells us this has a \((1-d)\)-shifted Poisson structure.

When will this be a symplectic structure?  We will need to show that, for every extension of \(F\) to \(F'\) on \(D+E\), the map
\[
\TT_{\Map(D+E,BGL_n),F'}\to\TT_{\Map(D,BGL_n),F}
\]
is a quasi-isomorphism.  Vector bundles on \(D+E=dD\) are governed by \(H^1(dD,GL_n(\calo_{dD}))\).  Let \(\calo_D(1)\cong\calo_D(D)\) be the relative hyperplane bundle of \(D\) over \(W\), and let \(i:D\to dD\) be the inclusion.  Then as a \(\calo_D\) module, we have \(i^*\calo_{dD}\simeq\bigoplus_{i=0}^{d-1}\calo_D(-i)\), so \(ker(\calo_{dD}\to\calo_D)\simeq\bigoplus_{i=1}^{d-1}\calo_D(-1)\) and similarly
\[
ker(GL_n(\calo_{dD})\to GL_n(\calo_D))\simeq\operatorname{Mat}_{n\times n}(\bigoplus_{i=1}^{d-1}\calo_D(-i)).
\]
Thus we get an exact sequence
\[
H^0(D,\operatorname{Mat}_{n\times n}(\oplus_{i=1}^{d-1}\calo_D(-i)))\to H^1(dD,GL_n(\calo_{dD}))\to H^1(D,GL_n(\calo_D)).
\]
But \(H^0(D,\calo_D(-i))\simeq0\) for \(1\leq i\leq d-1\), so the map \(H^1(dD,GL_n(\calo_{dD}))\to H^1(D,GL_n(\calo_D))\) is injective.  In particular, \(F\) has only one extension to \(dD\), namely \(F\otimes\calo_{dD}\).

It remains to see when \(\TT_{\Map(dD,BGL_n),F\otimes\calo_{dD}}\to\TT_{\Map(D,BGL_n),F}\) is a quasi-isomorphism.  We have \(\TT_{\Map(D,BGL_n),F}\simeq Ext^*(F,F)[1]\), and
\[
\TT_{\Map(dD,BGL_n),F\otimes\calo_{dD}}\simeq Ext^*(F\otimes\calo_{dD},F\otimes\calo_{dD})[1]\simeq\RR\Gamma(D,\mathcal{H}om(F,F)\otimes i^*\calo_{dD})[1].
\]
Thus we get a triangle
\[
\RR\Gamma(D,\mathcal{H}om(F,F)\otimes \bigoplus_{i=1}^{d-1}\calo_D(-i))[1]\to\TT_{\Map(D+E,BGL_n),F'}\to\TT_{\Map(D,BGL_n),F}.
\]
Thus we get a shifted symplectic structure if and only if
\[
H^*(\mathcal{H}om(F,F)\otimes\calo_D(-i))=0\quad 1\leq i\leq d-1.
\]

\end{document}